\documentclass[10pt]{article}

\usepackage[a4paper, total={6.5in, 8.2in}]{geometry}

\usepackage{amsmath}
\usepackage{hyperref}
\usepackage{mathtools}
\usepackage[makeroom]{cancel}
\usepackage{amssymb}
\usepackage{amsthm}
\usepackage{tensor}
\usepackage{enumerate}
\usepackage{cite}
\usepackage{cleveref}
\usepackage{bbm}
\usepackage{float}
\usepackage{varwidth}
\usepackage{pictexwd, dcpic}
\usepackage{xcolor}
\hypersetup{
    colorlinks,
    linkcolor={blue!60!black},
    citecolor={red!60!black},
    urlcolor={blue!80!black}
}
\usepackage{graphicx}
\usepackage{fancyhdr}
\usepackage{mathrsfs}
\usepackage{xfrac}
\usepackage[utf8]{inputenc}
\usepackage{tikz}
\usepackage{scalerel,stackengine}
\usepackage{booktabs}
\usepackage{slashed}

\theoremstyle{plain}
\newtheorem{theorem}{Theorem}
\newtheorem*{theorem*}{Theorem}
\newtheorem{lemma}[theorem]{Lemma}
\newtheorem{proposition}[theorem]{Proposition}
\newtheorem{corollary}[theorem]{Corollary}
\newtheorem*{corollary*}{Corollary}

\theoremstyle{definition}
\newtheorem{definition}[theorem]{Definition}

\theoremstyle{remark}
\newtheorem{remark}[theorem]{Remark}

\numberwithin{theorem}{section}
\numberwithin{equation}{section}

\renewcommand{\d}{\mathrm{d}}
\renewcommand{\leq}{\leqslant}
\renewcommand{\geq}{\geqslant}
\renewcommand{\epsilon}{\varepsilon}

\newcommand{\D}{\mathrm{D}}
\newcommand{\la}{\lesssim}

\newcommand{\scri}{\mathscr{I}}
\newcommand{\e}{\mathrm{e}}
\newcommand{\slashgrad}{\slashed{\nabla}}

\DeclareFontFamily{U}{MnSymbolC}{}
\DeclareSymbolFont{MnSyC}{U}{MnSymbolC}{m}{n}
\DeclareFontShape{U}{MnSymbolC}{m}{n}{
    <-6>  MnSymbolC5
   <6-7>  MnSymbolC6
   <7-8>  MnSymbolC7
   <8-9>  MnSymbolC8
   <9-10> MnSymbolC9
  <10-12> MnSymbolC10
  <12->   MnSymbolC12}{}
\DeclareMathSymbol{\intprod}{\mathbin}{MnSyC}{'270}

\newcommand*{\defeq}{\mathrel{\vcenter{\baselineskip0.5ex \lineskiplimit0pt
                     \hbox{\scriptsize.}\hbox{\scriptsize.}}}%
                     =}

\newcommand*{\eqdef}{=\mathrel{\vcenter{\baselineskip0.5ex \lineskiplimit0pt
                     \hbox{\scriptsize.}\hbox{\scriptsize.}}}%
                     }

\DeclareFontFamily{U}{BOONDOX-calo}{\skewchar\font=45 }
\DeclareFontShape{U}{BOONDOX-calo}{m}{n}{
  <-> s*[1.05] BOONDOX-r-calo}{}
\DeclareFontShape{U}{BOONDOX-calo}{b}{n}{
  <-> s*[1.05] BOONDOX-b-calo}{}
\DeclareMathAlphabet{\mathcalboondox}{U}{BOONDOX-calo}{m}{n}
\SetMathAlphabet{\mathcalboondox}{bold}{U}{BOONDOX-calo}{b}{n}
\DeclareMathAlphabet{\mathbcalboondox}{U}{BOONDOX-calo}{b}{n}

\DeclareMathOperator{\dvol}{dv}

\title{Large data decay of Yang--Mills--Higgs fields \\ on Minkowski and de Sitter spacetimes}

\author{Grigalius Taujanskas\footnote{Electronic address: \texttt{taujanskas@maths.ox.ac.uk}. This work was supported by the Engineering and Physical Science Research Council [EP/L05811/1].} \vspace{0.3cm} \\ \small \textsc{Mathematical Institute} \\ \small \textsc{Oxford University} \\ \small \textsc{Radcliffe Observatory Quarter} \\ \small \textsc{Oxford OX2 6GG, UK}}

\begin{document}

\maketitle

\abstract{In this article we extend Eardley and Moncrief's $L^\infty$ estimates \cite{EardleyMoncrief1982b} for the conformally invariant Yang--Mills--Higgs equations to the Einstein cylinder. Our method is to first work on Minkowski space and localize their estimates, and then carry them to the Einstein cylinder by a conformal transformation. By patching local estimates together, we deduce global $L^\infty$ estimates on the cylinder, and extend Choquet-Bruhat and Christodoulou's \cite{ChoquetBruhatChristodoulou1981} small data well-posedness result to large data. Finally, by employing another conformal transformation, we deduce exponential decay rates for Yang--Mills--Higgs fields on \mbox{de Sitter} space, and inverse polynomial decay rates on Minkowski space.}

\section{Introduction}

The analytical study of classical Yang--Mills--Higgs equations goes back to at least the late 1970s, with Segal's local existence proof \cite{Segal1979,Segal1963} on Minkowski space of pure $\mathrm{SU}(2)$ Yang--Mills fields with $H^3 \times H^2$ initial data. A short time after Segal's proof, in 1981, Ginibre and Velo \cite{GinibreVelo1981}, Choquet-Bruhat and Christodoulou \cite{ChoquetBruhatChristodoulou1981}, and Eardley and Moncrief \cite{EardleyMoncrief1982a,EardleyMoncrief1982b} all published proofs of similarly major results, though using profoundly different techniques. Ginibre and Velo's work extended Segal's work to coupled Yang--Mills--Higgs equations in arbitrary dimension, in particular proving global existence in two and three spacetime dimensions. In four dimensions, Choquet-Bruhat and Christodoulou made use of the conformal invariance of the Yang--Mills--Higgs--Dirac equations and a short-time existence theorem on the Einstein cylinder to prove the global existence of solutions on Minkowski space for sufficiently small $H^2 \times H^1$ initial data (\emph{cf.} \cite{ChoquetBruhatPaneitzSegal1983}). Eardley and Moncrief, on the other hand, instead developed a physical space technique for extracting remarkable a priori estimates that allowed them to prove the global existence of solutions for \emph{large} $H^2 \times H^1$ initial data. A short time later, Goganov and Kapitanskii published a proof of global unique solvability \cite{GoganovKapitanskii1985} for only locally $H^2 \times H^1$ data on Minkowski space, in particular allowing arbitrary magnetic charge at spatial infinity. Their proof in particular shows that the equations are well-posed in local lightcones, with solutions determined only by the data at the base of the lightcone. Further improvements have been obtained by Klainerman and Machedon \cite{KlainermanMachedon1993,KlainermanMachedon1994,KlainermanMachedon1995} and others \cite{Tao2003,SelbergTesfahun2013,SungJinOh2015}.

The strategy of Eardley and Moncrief is to write down wave equations for the fields $F$ and $\D \phi$, treat the nonlinear terms in these equations as sources, and express $F$ and $\D \phi$ at a point $p$ as integrals over the backward lightcone of $p$. Their key observation is that these lightcone integrals can be estimated by expressions of the form
\[ E_0 \int_0^t \left( \| F(s) \|_{L^\infty} + \| \D \phi(s) \|_{L^\infty} \right) \d s, \]
which implies, via Gr\"onwall's inequality, that the $L^\infty$ norms cannot blow up in finite time. Part of the trick is to define the $L^\infty$ norms in a gauge-independent manner, and use the Cr\"onstrom gauge in intermediate calculations. Equipped with this estimate, it is then straightforward to show that the $H^2 \times H^1$ norm of the solution does not blow up in finite time. An incarnation of this method has been adapted, for pure Yang--Mills equations, to arbitrary smooth globally hyperbolic four dimensional spacetimes by Chru\'sciel and Shatah \cite{ChruscielShatah1997}, by replacing the lightcone integrals with Friedlander's representation formula \cite{Friedlander1975} for the covariant wave equation. However, Chru\'sciel and Shatah require effectively $H^3 \times H^2$ data to deal with a term that causes difficulties in curved space\footnote{Eardley and Moncrief's result has been re-proven by Klainerman and Rodnianski by applying their newly developed Kirchoff--Sobolev parametrix for the wave equation \cite{KlainermanRodnianski2006}. A similar method has since been used by Ghanem \cite{Ghanem2016} to give another proof of the a priori estimates for pure Yang--Mills on curved spacetimes.}. Though the system has been well-studied, Eardley and Moncrief's method with $H^2 \times H^1$ data for coupled Yang--Mills and Higgs equations does not seem to have been explicitly adapted to curved space, even in the case of the Einstein cylinder. The scalar field part scales differently under a conformal transformation, putting it on unequal footing with the Yang--Mills potential. In particular, this upsets the conformal invariance of the system somewhat, breaking the invariance of the canonical energy-momentum tensor. And although formally the field equations remain conformally invariant, the scalar field introduces a boundary term in the conformal variation of the action that has a non-trivial dependence on the decay of the scalar field. This is expected to be of some importance in path integral formulations of interacting quantum field theories.

In this article we extend the $L^\infty$ estimates of Eardley and Moncrief to the Einstein cylinder. Our method is inspired by and combines the techniques of \cite{GoganovKapitanskii1985,ChoquetBruhatChristodoulou1981,EardleyMoncrief1982b}: we first work on Minkowski space and localize Eardley and Moncrief's estimates, removing the requirement of the global finiteness of the energy. Then, using a conformal transformation, we glue a small conical patch of Minkowski space onto the Einstein cylinder, and show that $L^\infty$ estimates in the Minkowskian patch imply local $L^\infty$ estimates on the cylinder. By patching a finite number of such cones all the way around the Einstein cylinder, we deduce $L^\infty$ bounds on any finite section of the cylinder. This allows us to show that Choquet-Bruhat and Christodoulou's small data result on the Einstein cylinder \cite{ChoquetBruhatChristodoulou1981} extends to large data, and consequently removes the small data restriction in the scattering theory of \cite{Taujanskas2018}. Finally, by using another conformal transformation, we deduce large data decay rates of Yang--Mills--Higgs fields on Minkowski and de Sitter spacetimes.

The structure of the paper is as follows. In \Cref{sec:YMH} we recall briefly the definitions of the Yang--Mills--Higgs fields and the associated field equations, as well as their conformal properties. In \Cref{sec:conventions} we introduce the conventions and notation that will be used throughout the paper. In \Cref{sec:localization} we sketch the method of Eardley and Moncrief and show that their estimates can be localized along the way. In \Cref{sec:gluing,sec:globalexistence} we glue the Minkowskian $L^\infty$ estimates onto the Einstein cylinder and use them deduce the global existence of Yang--Mills--Higgs fields on $\mathbb{R} \times \mathbb{S}^3$. Finally, in \Cref{sec:asymptotics} we deduce decay rates for the fields on Minkowski space and de Sitter space.

\textbf{Acknowledgements.} The author would like to thank Qian Wang and Lionel Mason for discussions which inspired this work, and Paul Tod and Jan Sbierski for valuable feedback.

\section{The Yang--Mills--Higgs Equations} \label{sec:YMH}

Let $\mathrm{G}$ be a connected matrix Lie group with a compact semi-simple Lie algebra $\mathfrak{g}$. In particular, we assume that $\mathfrak{g}$ is represented by a subalgebra of the algebra of real matrices equipped with the usual matrix commutator, and admits a positive-definite $\mathrm{Ad}$-invariant scalar product $\langle\cdot, \cdot\rangle$ given by
\[ \langle X,Y\rangle = - \operatorname{Tr}(XY) \quad \forall X, Y \in \mathfrak{g}. \]
 Let $\{ \theta_\alpha \}$ be the generators of $\mathfrak{g}$ in such a representation and let $f_{\alpha \beta}^{\phantom{\alpha \beta}\gamma}$ be the structure constants of $\mathfrak{g}$, so that
\[ [ \theta_{\alpha}, \theta_{\beta}] = f_{\alpha \beta}^{\phantom{\alpha \beta} \gamma} \theta_{\gamma}. \]
Since $\mathfrak{g}$ is semi-simple, $f_{\alpha \beta \gamma}$ can be chosen to be totally antisymmetric in its indices. Furthermore, the generators can be chosen to be real antisymmetric matrices satisfying
\[ \langle \theta_\alpha , \theta_\beta \rangle = \delta_{\alpha \beta}. \]
Let $\mathcalboondox{M}$ be a globally hyperbolic four-dimensional spacetime and let $P \to \mathcalboondox{M}$ be a principal $\mathrm{G}$-bundle over $\mathcalboondox{M}$. By global hyperbolicity, $\mathcalboondox{M}$ admits a Cauchy surface $\Sigma$ and a global smooth time function $t$ such that $\Sigma = \{ t = 0 \}$. Furthermore, the flow along the integral curves of the gradient of $t$ defines a diffeomorphism $\mathcalboondox{M} \simeq \mathbb{R} \times \Sigma$. If $P^\Sigma$ is the pullback of $P$ to $\Sigma$, this leads to the identification of $P \to \mathcalboondox{M}$ with $\mathbb{R} \times P^\Sigma \to \mathbb{R} \times \Sigma$.

The Yang--Mills potential $A$ is a connection on $P$, and in any trivialization of $P$ over a coordinate patch $\mathcalboondox{U}$ of $\mathcalboondox{M}$ is given by a $\mathfrak{g}$-valued $1$-form on $\mathcalboondox{U}$,
\[ A = A_a(x) \, \d x^a, \qquad A_a(x) = A_a^\alpha(x) \theta_\alpha \in \mathfrak{g} \]
for some real-valued functions $A^\alpha_a$ on $\mathcalboondox{U}$. The curvature of $A$ (or the Yang--Mills field) is the $\mathfrak{g}$-valued $2$-form 
\[F = F_{ab}(x) \, \d x^a \wedge \d x^b = ( F_{ab}^\alpha(x) \theta_\alpha) \, \d x^a \wedge \d x^b \]
given by 
\[ F_{ab} = \nabla_a A_b - \nabla_b A_a + [ A_a, A_b] \]
in $\mathcal{U}$, where $\nabla_a$ is the Levi--Civita connection on $\mathcalboondox{M}$. We shall denote the projections of $A$ and $F$ onto $\Sigma_t = \{ t = \text{const.} \} \times \Sigma$ by $\mathbf{A}_i$ and $F_{ij}$ respectively, where the indices $i$ and $j$ will range over $\{ 1,2,3 \}$, and $x^0 = t$. We define the Higgs field $\phi$ to be a section of the real vector bundle associated to the representation $\{ \theta_\alpha \}$. We will denote the inner product of such sections by $\phi \cdot \psi = \phi_\alpha \psi_\alpha$, and write, for example, $|\phi|^2 = \phi \cdot \phi = \phi_\alpha \phi_\alpha$. The gauge-covariant derivative $\D_a$ of $\phi$ is defined to be
\[ \D_a \phi = \nabla_a \phi + A_a \phi. \] Under a gauge transformation
\begin{align*} A_a &\leadsto U A_a U^{-1} + U \partial_a U^{-1}, \\
F_{ab} & \leadsto U F_{ab} U^{-1}, \\
\phi & \leadsto U \phi, \text{ and} \\
\D_a \phi & \leadsto U \D_a \phi
\end{align*}
for any smooth $\mathrm{G}$-valued function $U$ on $\mathcalboondox{M}$.

The conformally invariant Lagrangian for Yang--Mills--Higgs on $(\mathcalboondox{M}, g)$ is 
\begin{equation} \label{Lagrangian} \mathcal{L} = - \frac{1}{4} \langle F_{ab}, F^{ab} \rangle + \frac{1}{2} (\D_a \phi) \cdot (\D^a \phi) - \frac{1}{12} R | \phi|^2 - \frac{1}{4} \lambda |\phi|^4,	
\end{equation}
where $R$ is the scalar curvature of $g$ and $\lambda \geq 0$ is a constant. The Euler-Lagrange equations associated to \eqref{Lagrangian} are
\begin{equation} \label{fieldequations} \D^b F_{ab} = - ( (\D_a \phi) \cdot \theta_\alpha \phi ) \theta_\alpha \quad \text{and} \quad \D^a \D_a \phi + \frac{1}{6} R \phi + \lambda |\phi|^2 \phi = 0
\end{equation}
where
\[ \D_a F_{bc} = \nabla_a F_{bc} + [A_a, F_{bc}]. \]
The curvature also obeys the Bianchi identity
\begin{equation} \label{BianchiF} \D_{[a} F_{bc]} = 0. \end{equation}
Note that by virtue of $\lambda \geq 0$, the equation for $\phi$ is a \emph{defocussing} nonlinear wave equation. The stress-energy tensor for \eqref{Lagrangian} obtained by variation with respect to the metric is
\begin{align} \begin{split} \label{stresstensor} \mathbf{T}_{ab} & = 2 \frac{\delta \mathcal{L}}{\delta g^{ab}} - g_{ab} \mathcal{L} \\
& = - \langle F_{ac}, F_b^{\phantom{b}c} \rangle + \frac{1}{4} g_{ab} \langle F_{cd}, F^{cd} \rangle + (\D_{a} \phi) \cdot (\D_{b} \phi) - \frac{1}{2} g_{ab} (\D_c \phi) \cdot (\D^c \phi) - \frac{1}{6} G_{ab} |\phi|^2 + \frac{1}{4} \lambda g_{ab} |\phi|^4. 
\end{split}
\end{align}
As a consequence of the field equations \eqref{fieldequations}, $\mathbf{T}_{ab}$ satisfies the approximate conservation law 
\[ \nabla^a \mathbf{T}_{ab} = - \frac{1}{3} R_{ab} \phi \cdot \nabla^a \phi. \] 
It can be checked that the equations \eqref{fieldequations} are conformally invariant. That is, under the conformal transformation of the metric
\[ \hat{g}_{ab} = \Omega^2 g_{ab} \]
for some function $\Omega > 0$, the rescaled fields
\[ \hat{A}_a = A_a, \qquad \hat{F}_{ab} = F_{ab}, \qquad \hat{\phi} = \Omega^{-1} \phi \]
satisfy the rescaled field equations
\[ \hat{\D}^b \hat{F}_{ab} = - ( (\hat{\D}_a \hat{\phi}) \cdot \theta_\alpha \hat{\phi} )\theta_\alpha, \qquad \hat{\D}^a \hat{\D}_a \hat{\phi} + \frac{1}{6} \hat{R} \hat{\phi} + \lambda |\hat{\phi}|^2 \hat{\phi} = 0 \]
 if and only if the physical fields satisfy the equations \eqref{fieldequations}. Here $\hat{\D}_a \hat{\phi} = \hat{\nabla}_a \hat{\phi} + \hat{A}_a \hat{\phi}$, $\hat{\D}_a \hat{F}_{bc} = \hat{\nabla}_a \hat{F}_{bc} + [ \hat{A}_a, \hat{F}_{bc} ]$, $\hat{\nabla}_a$ is the Levi--Civita connection of $\hat{g}_{ab}$, and $\hat{R}$ is the scalar curvature of $\hat{g}_{ab}$. Note, however, that the stress-energy tensor \eqref{stresstensor} is \emph{not} conformally invariant. This is effectively due to the presence of the Higgs field $\phi$.

\section{Conventions and Notation} \label{sec:conventions}

Consider $(\mathcalboondox{M},g)$ a globally hyperbolic four dimensional Lorentzian manifold of signature $(+,-,-,-)$. We choose a global smooth time function $t$ such that $\nabla^a t$ is uniformly timelike on $\mathcalboondox{M}$, and assume that the metric $g$ takes the form
\[ g_{ab} = T_a T_b - h_{ab}, \quad \text{i.e.} \quad g = N^2 \d t^2 - h, \]
where $T^a$ is a smooth future-oriented uniformly timelike vector field with lapse function $N(t)$, $T^a = N^{-1} \partial_t$, and $h_{ab}$ is a smooth Riemannian metric for each fixed $t$. The vector field $T^a$ defines a foliation of $\mathcalboondox{M}$ by hypersurfaces $\Sigma_t$ of constant $t$, and identifies $\mathcalboondox{M} = \mathbb{R} \times \Sigma$, where each $\Sigma_t$ is diffeomorphic to $\Sigma$. We denote by $\nabla_a$ the Levi--Civita connection of $g$, and define the Sobolev spaces on the hypersurfaces $\Sigma_t$ with respect to the Riemannian metric $h(t)$. To be able to work with Sobolev spaces in spacetime, we define the four dimensional Riemannian metric
\[ \Gamma_{ab} \defeq 2 T_a T_b - g_{ab} = T_a T_b + h_{ab}, \]
and define Sobolev norms on general subsets of $\mathcalboondox{M}$ with respect to $\Gamma$. For example, for a matrix-valued $2$-form $F_{ab} = F_{ab}^\alpha \theta_\alpha$ on $\mathcalboondox{M}$ we set
\[ | F |^2_{\Gamma} \defeq \sum_\alpha F_{ab}^\alpha F^\alpha_{cd} \Gamma^{ac} \Gamma^{bd}, \]
and define
\[ \| F \|_{L^\infty(K)} \defeq \sup_K |F|_{\Gamma}  \]
for any $K \subset \mathcalboondox{M}$.

We will specifically work on three conformally related\footnote{By conformally related here we mean that there exist smooth non-negative functions $\Omega$ and $\omega$ such that $\mathfrak{e} = \Omega^2 \eta$ and $\mathfrak{e} = \omega^2 \tilde{g}$.} spacetimes of the above form: Minkowski space $(\mathbb{M} = \mathbb{R}^4, \eta)$, where 
\[ \eta = \d t^2 - \d r^2 - r^2 \mathfrak{s}_2, \]
the Einstein cylinder $(\mathfrak{E} = \mathbb{R} \times \mathbb{S}^3, \mathfrak{e})$, where
\[ \mathfrak{e} = \d \tau^2 = \mathfrak{s}_3, \]
and de Sitter space $(\mathrm{dS}_4 = \mathbb{R} \times \mathbb{S}^3, \tilde{g})$, where
\[ \tilde{g} = \d \alpha^2 - ( \cosh^2 \alpha) \mathfrak{s}_3. \]
Here $\mathfrak{s}_n$ represents the standard Riemannian metric on $\mathbb{S}^n$. Unless stated otherwise, we will denote the Levi--Civita connection on $\mathbb{M}$ by $\nabla_a$, the Levi--Civita connection on $\mathfrak{E}$ by $\hat{\nabla}_a$, and the Levi--Civita connection on $\mathrm{dS}_4$ by $\tilde{\nabla}_a$. We also denote the Levi--Civita connection on $\mathbb{R}^3$ by $\boldsymbol{\nabla}$ and the Levi--Civita connection on $\mathbb{S}^3$ by $\slashgrad$. In each of the three spacetimes one has a standard uniformly timelike vector field: $\partial_t$ in $\mathbb{M}$, $\partial_\tau$ in $\mathfrak{E}$, and $\partial_\alpha$ in $\mathrm{dS}_4$. We shall use these to define foliations of $\mathbb{M}$, $\mathfrak{E}$ and $\mathrm{dS}_4$, as described above. Given a solution $(A_a, \phi)$ to the Yang--Mills--Higgs equations on Minkowski space, we will denote the corresponding conformally related solution on the Einstein cylinder by $(\hat{A}_a, \hat{\phi})$, and the corresponding solution on de Sitter space by $(\tilde{A}_a, \tilde{\phi})$. The timelike components (corresponding to the time coordinate in each spacetime) of the Yang--Mills potential will be denoted with the index $0$, \emph{i.e.} $A_0 = (\partial_t)^a A_a$, $\hat{A}_0 = (\partial_\tau)^a \hat{A}_0$, and $\tilde{A}_0 = (\partial_\alpha)^a \tilde{A}_a$. We will denote by $\mathbf{A}$ (or $\hat{\mathbf{A}}$, or $\tilde{\mathbf{A}}$) the projection of $A$ onto the spacelike slice $\Sigma_t$ (or $\Sigma_\tau$, or $\Sigma_\alpha$ respectively), and define the electric and magnetic fields $\mathbf{E}$ and $\mathbf{B}$ on $\mathbb{M}$ by
\[ \mathbf{E}_i = F_{0i}, \quad \text{and} \quad \mathbf{B}^i = \frac{1}{2} \epsilon^{ijk} F_{jk}. \]
The electric and magnetic fields on $\mathfrak{E}$ and $\mathrm{dS}_4$ are defined similarly, and denoted $\hat{\mathbf{E}}$ and $\hat{\mathbf{B}}$, and $\tilde{\mathbf{E}}$ and $\tilde{\mathbf{B}}$ respectively. Here the Roman indices $i,j,k$ run over $\{1,2,3\}$ and denote contractions with the spatial basis vectors $\partial_i = \partial/\partial x^i$, $i=1,2,3$. We also define
\[ \pi = \D_0 \phi, \]
where $\D_a \phi  = \nabla_a \phi + A_a \phi$, and similarly define $\hat{\pi}$ and $\tilde{\pi}$. In intermediate calculations we shall want to manipulate the components of the Yang--Mills field $F_{ab}$ relative to a null tetrad $(l,n,e_A)$, $A \in \{ \theta, \phi \}$, and will denote by $F_{ln} = l^a F_{ab} n^b$, $F_{lA} = l^a F_{ab} (e_A)^b$, and so on.

Finally, in the analysis we shall use the letter $C$ to denote a constant that may change from line to line, and $p(t)$ to denote an arbitrary positive ``generalized" polynomial in $t$ perhaps containing positive fractional powers of $t$.

\section{Localized \texorpdfstring{$L^\infty$}{Linfty} Estimates on Minkowski Space} \label{sec:localization}

On Minkowski space the field equations \eqref{fieldequations} simplify to 
\begin{equation} \label{Minkowskifieldequations} \D^b F_{ab} = - ( (\D_a \phi) \cdot \theta_\alpha \phi) \theta_\alpha \quad \text{and} \quad \D^a \D_a \phi + \lambda |\phi|^2 \phi = 0.	
\end{equation}
In temporal gauge $A_0 = 0$ they further split into
\begin{equation} \label{Minkowskifieldequationstemporalgauge} \dot{\mathbf{E}}_i + \boldsymbol{\nabla}_j F_{ij} + [\mathbf{A}_j, F_{ij}] = ( (\boldsymbol{\D}_i \phi) \cdot \theta_\alpha \phi) \theta_\alpha, \qquad \dot{\pi} - \boldsymbol{\D}_i \boldsymbol{\D}_i \phi + \lambda |\phi|^2 \phi = 0, \end{equation}
and the constraint equation
\begin{equation} \label{Minkowskiconstraint} \boldsymbol{\nabla} \cdot \mathbf{E} + [\mathbf{A}_i, \mathbf{E}_i] = ( \pi \cdot \theta_\alpha \phi ) \theta_\alpha. \end{equation}
Of course, the constraint \eqref{Minkowskiconstraint} is propagated in the sense that it is satisfied for all time if it is satisfied initially. We will ultimately consider the system \eqref{Minkowskifieldequationstemporalgauge}--\eqref{Minkowskiconstraint}, but shall use the Cronstr\"om gauge to derive the intermediate a priori $L^\infty$ estimates.

By differentiating the Bianchi identity \eqref{BianchiF} and using the field equations \eqref{Minkowskifieldequations}, one derives a wave equation for the curvature $F$, which turns out to be
\begin{align} \begin{split} \label{waveeqnF} \Box F_{ab} &= \left((F_{ab} \phi) \cdot \theta_\alpha \phi \right) \theta_\alpha + \left( (\D_b \phi) \cdot \theta_\alpha (\D_a \phi) - (\D_a \phi) \cdot \theta_\alpha (\D_b \phi ) \right) \theta_\alpha \\ 
&- 2 \nabla^c \left( [A_c, F_{ab}] \right)	+ [\nabla_c A^c, F_{ab}] - [A^c, [A_c, F_{ab}]] - 2[F_b^{\phantom{b}c}, F_{ac}].
\end{split}	
\end{align}
By differentiating the wave equation for $\phi$ and using the field equation for $F$, one also derives
\[ \D^a \D_a (\D_b \phi) = \left( (\D_b \phi) \cdot \theta_\alpha \phi \right) \theta_\alpha \phi - 2 F_b^{\phantom{b}a} \D_a \phi - \lambda \D_b (|\phi|^2 \phi ), \]
which can be written as
\begin{equation} \label{waveeqnDphi}
	\begin{split} \Box (\D_b \phi) &= - 2 \nabla^a (A_a \D_b \phi) + (\nabla^a A_a) \D_b \phi - A^a A_a \D_b \phi \\
	& + \left( (\D_b \phi) \cdot \theta_\alpha \phi \right) \theta_\alpha \phi - 2 F_b^{\phantom{b}a} \D_a \phi - \lambda \D_b (|\phi|^2 \phi ).
	\end{split}
\end{equation}
Here $\Box$ denotes the standard wave operator on Minkowski space. It is worth observing that temporal gauge initial data $(\mathbf{A}, \mathbf{E}, \phi, \pi)$ for the equations \eqref{Minkowskifieldequationstemporalgauge} defines initial data for the wave equations \eqref{waveeqnF} and \eqref{waveeqnDphi}. Indeed, the data for $F$ is given by
\begin{align*} & F_{0i}|_{t=0} = \mathbf{E}_i, \qquad \partial_t F_{0i}|_{t=0} = - \boldsymbol{\nabla}_j F_{ij} - [ \mathbf{A}_j, F_{ij}] + \left( (\boldsymbol{\D}_i \phi ) \cdot \theta_\alpha \phi \right) \theta_\alpha, \\
& F_{ij}|_{t=0} = \boldsymbol{\nabla}_i \mathbf{A}_j - \boldsymbol{\nabla}_j \mathbf{A}_i + [\mathbf{A}_i, \mathbf{A}_j],  \qquad \partial_t F_{ij}|_{t=0} = \boldsymbol{\nabla}_i \mathbf{E}_j - \boldsymbol{\nabla}_j \mathbf{E}_i + [\mathbf{E}_i, \mathbf{A}_j] + [\mathbf{A}_i, \mathbf{E}_j],
\end{align*}
while data for $\D \phi$ is given by
\begin{align*} & \D_0 \phi|_{t=0} = \pi, \qquad \partial_t (\D_0 \phi)|_{t=0} = \boldsymbol{\D}_i \boldsymbol{\D}_i \phi - \lambda |\phi|^2 \phi, \\
& \boldsymbol{\D}_i \phi|_{t=0} = \boldsymbol{\nabla}_i \phi + \mathbf{A}_i \phi, \qquad \partial_t (\boldsymbol{\D}_i \phi)_{t=0} = \boldsymbol{\nabla}_i \pi + \mathbf{E}_i \phi + \mathbf{A}_i \pi.
\end{align*}
We will use the wave equations \eqref{waveeqnF} and \eqref{waveeqnDphi} to write down integral expressions for $F$ and $\D \phi$, which will be crucial for our analysis. Before we do that, however, we need a couple of preliminary tools.

\subsection{Conservation of Energy}

In standard coordinates on Minkowski space, the vector field $\partial_t$ is a globally defined uniformly timelike Killing field. Furthermore, the stress-energy tensor \eqref{stresstensor} is conserved, and becomes
\[ \mathbf{T}_{ab} = - \langle F_{ac}, F_b^{\phantom{b}c} \rangle + \frac{1}{4} \eta_{ab} \langle F_{cd}, F^{cd} \rangle + (\D_a \phi)\cdot (\D_b \phi) - \frac{1}{2} \eta_{ab} (\D_c \phi) \cdot (\D^c \phi) + \frac{1}{4} \lambda \eta_{ab} |\phi|^4. \]
Contracting $\mathbf{T}_{ab}$ with the Killing field $(\partial_t)^b$ defines a conserved current whose timelike component is
\[ \mathbf{T}_{00} = \frac{1}{2} \langle \mathbf{E}_i, \mathbf{E}_i \rangle + \frac{1}{2} \langle \mathbf{B}_i, \mathbf{B}_i \rangle + \frac{1}{2} \pi \cdot \pi + \frac{1}{2} (\boldsymbol{\D}_i \phi) \cdot (\boldsymbol{\D}_i \phi) + \frac{1}{4} \lambda |\phi|^4.  \]
It follows that the energy
\[ E_0(t) = \frac{1}{2} \int_{\mathbb{R}^3} \left( | \mathbf{E} |^2 + |\mathbf{B}|^2 + |\pi|^2 + | \boldsymbol{\D} \phi |^2 + \frac{1}{2} \lambda |\phi|^4 \right) \d^3 x \]
is conserved, where $|\mathbf{E}|^2 = \langle \mathbf{E}_i, \mathbf{E}_i \rangle$, and so on.

\begin{remark} \label{rmk:energycharges} In view of the conformal compactification of Minkowski space, note that generic $H^1(\mathbb{S}^3) \times L^2(\mathbb{S}^3)$ initial data on the Einstein cylinder will render $E_0 = \infty$ on Minkowski space, due to the unavoidable introduction of charges. Since Eardley and Moncrief's $L^\infty$ estimates rely on $E_0$ being finite (\emph{c.f.} \S1), this is the primary reason we need to make sure they can be localized. See \Cref{rmk:infiniteenergy} for more details.
\end{remark}

More generally, one may contract $\mathbf{T}_{ab}$ with any timelike Killing field $K^a$ to get a conserved current
\[ J_a \defeq \mathbf{T}_{ab} K^b, \]
and derive energy identities by integrating the identity $\nabla_a J^a = 0$ over bounded regions of spacetime. We will do so shortly to derive an energy identity on a lightcone. To do this, we equip ourselves with the following basis of vector fields,
\[ l^a = -\partial_t + \partial_r, \quad n^a = \partial_t + \partial_r, \quad e^a_\theta = \frac{1}{r} \partial_\theta, \quad e^a_\phi = \frac{1}{r \sin \theta} \partial_\phi.  \]
The vector fields $(l,n,e_A)$, $A \in \{\theta, \phi \}$, satisfy
\[ l_a l^a = 0 = n_a n^a, \quad l_a n^a = -2, \quad (e_A)_a (e_B)^a = - \delta_{AB}, \]
and the Minkowski metric can be written in terms of the basis $(l,n,e_A)$ as
\[ \eta_{ab} = - \frac{1}{2} (l_a n_b + l_b n_a) + (e_A)_a (e_A)_b, \]
where the index $A$ is summed over $\{\theta, \phi \}$. Similarly, the volume form can be written as
\[ \d t \wedge \d^3 x = \frac{1}{2} l^\flat \wedge n^\flat \wedge e_\theta^\flat \wedge e_\phi^\flat. \]
Putting $K^a = \partial_t$ and integrating $\nabla_a J^a = 0$ over the region bounded by the past lightcone of the origin $K = \{ t = -r \}$ and the surface $\Sigma = \{ t = - t_0 \}$, $t_0 > 0$, we get
\[ \frac{1}{2} \int_{B(r_0)} \left( | \mathbf{E} |^2 + | \mathbf{B} |^2 + |\pi|^2 + | \boldsymbol{\D} \phi|^2 + \frac{1}{2} \lambda |\phi|^4 \right) \d^3 x = - \int_{K(t_0)} (J_a l^a)(-r, r, \omega) \, r^2 \, \d r \, \d \Omega,  \]
where $K(t_0)$ is the past lightcone of the origin up to $t=-t_0$, and $B(r_0)$ is the solid ball in $\Sigma$ of radius $r_0 = t_0$. Expressing $K^a = \frac{1}{2}(n^a - l^a)$, we have
\begin{align} \begin{split} \label{energyidentitycone} & \frac{1}{2} \int_{B(r_0)} \left( | \mathbf{E} |^2 + | \mathbf{B} |^2 + |\pi|^2 + | \boldsymbol{\D} \phi|^2 + \frac{1}{2} \lambda |\phi|^4 \right) \d^3 x \\
& = \frac{1}{2} \int_{K(t_0)} \left( \frac{1}{4} |F_{ln}|^2 + |F_{lA}|^2 + \frac{1}{2} |F_{AB}|^2 + | \D_l \phi |^2 + |\D_A \phi |^2 + \frac{1}{2} \lambda |\phi|^4 \right) (-r,r,\omega) \, r^2 \, \d r \, \d \Omega .
\end{split}
\end{align}
We shall denote the left-hand side of the energy identity \eqref{energyidentitycone}, the energy in $B(r_0)$ at time $-t_0$, by $E_{B(r_0)}(-t_0)$.

\begin{definition} \label{defn:energyidentity} We define the \emph{local energy} $E_{\mathrm{loc}}(p)$ of a point $p=(t,x)$ by
	\begin{align*} E_{\mathrm{loc}}(p) &\defeq \sup_{s \in [0,t]} \frac{1}{2} \int_{B(x,t-s)} \left( |\mathbf{E}|^2 + |\mathbf{B}|^2 + |\pi|^2 + | \boldsymbol{\D} \phi|^2 + \frac{1}{2} \lambda |\phi|^4 \right) \d^3 x(s) \\
	& = \sup_{s \in [0,t]} E_{B(x,t-s)}(s),
	\end{align*}
	where $B(x,r)$ is the ball of radius $r$ centred at $x \in \mathbb{R}^3$.
	\begin{figure}[H]
\centering
	\begin{tikzpicture}
	\centering
	\node[inner sep=0pt] (localenergy) at (3.4,0)
    	{\includegraphics[width=.25\textwidth]{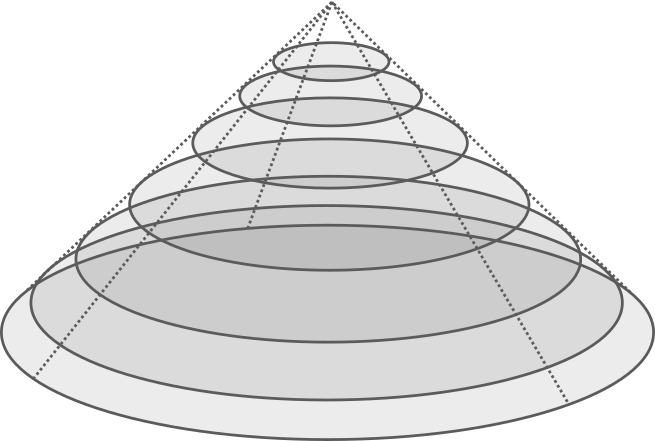}};

	\node[label={[shift={(6.5,-0.1)}]$t$}] {};
	\node[label={[shift={(3.4,1.5)}]$p$}] {};
	\draw[<->] (6.2,-.95) .. controls (6.2,0) .. (6.2,1.35);
	
	\end{tikzpicture}	
\end{figure}
\end{definition}

\begin{remark} It will be important to know that the local energy $E_{\mathrm{loc}}$ will be finite for $H^2_{\mathrm{loc}} \times H^1_{\mathrm{loc}}$ initial data for \eqref{Minkowskifieldequations}, for example by the results of Goganov--Kapitanskii \cite{GoganovKapitanskii1985}.	
\end{remark}

\subsection{The Cronstr\"om Gauge}

If $K(p)$ is the backwards lightcone from $p$ to the initial surface $\Sigma$ as above, we can choose an open set $S_p$ containing the set bounded by $K(p)$ and $\Sigma$ and impose the Cronstr\"om gauge in $S_p$. The Cronstr\"om gauge is defined by
\begin{equation} \label{Cronstromgauge} (x^a - x^a_p) A_a(x) = 0 \quad \text{and} \quad A_a(x_p) = 0 \quad \text{in } S_p,	
\end{equation}
and it can be shown \cite{EardleyMoncrief1982b} that on Minkowski space a given pair of fields $(A,\phi)$ can always be transformed to the Cr\"onstrom gauge in any star-shaped region (within the domain of existence of the solution). Furthermore, the associated gauge transformation is trivial at $p$, $U(x_p) = \mathbbm{1}$. An extremely useful feature of the Cronstr\"om gauge is that it allows one to express the Yang--Mills potential $A_a$ entirely in terms of the field $F_{ab}$. If we translate the origin to the point $p$ as before, one has
\begin{equation} \label{Cronstrompotential} A_b(x) = \int_0^1 s x^a F_{ab}(sx) \, \d s.
\end{equation}
From this one also derives
\begin{equation} \label{CronstromdivA} (\nabla_a A^a)(x) = \int_0^1 \left( s^2 x^a [F_{ab}(sx), A^b(sx)] - s^2 x^a \left( (\D_a \phi)(sx) \cdot \theta_\alpha \phi(sx) \right) \theta_\alpha \right) \d s.
\end{equation}
In the following estimates we will translate an arbitrary point $p = (t_0,x_0)$ to the origin for convenience, so that the initial data will sit at $\{ t = -t_0 \}$.  We will also write $E_{\mathrm{loc}}$ to denote $E_{\mathrm{loc}}(0)$, the local energy of the origin, where the lightcone considered will be of height $t_0$ to make contact with the initial data.

\subsection{Integral Representations and Localization}

We recall that on Minkowski space $(\mathbb{R}^4, \eta)$, $\eta = \d t^2 - \d r^2 - r^2 \mathfrak{s}_2$, the retarded Green's function $G$ for the wave operator $\Box$ is given by
\[ G(t,r) = \frac{1}{4 \pi r} \delta(t-r), \]
so that any solution $u$ to $\Box u = f$ can be written as
\[ u(t_0,x_0) = u^{(0)}(t_0,x_0) + (G*f)(t_0,x_0), \]
where $u^{(0)}$ is the solution to the free wave equation $\Box u^{(0)} = 0$ determined by the data for $u$. The convolution $G * f$ can be expressed as an integral over the past lightcone of $p = (t_0,x_0)$: translating $(t_0,x_0)$ to the origin for simplicity, we have
\begin{align*} (G*f)(0) &= \int_{\mathbb{R}} \d t \int_{\mathbb{R}^3} r^2 \, \d r \, \d \Omega \, G(-t,-x)f(t,x) \\
& = \int_{\mathbb{R}} \d t \int_{\mathbb{R}^3} r^2 \, \d r \, \d \Omega \, \frac{1}{4 \pi r} \delta(t+r) f(t,x) \\
& = \frac{1}{4 \pi} \int_K r \, \d r \, \d \Omega \, f(-r,x),
\end{align*}
where $K$ is the past lightcone of the origin. 

Suppose $p$ is a point in the domain of local existence of some solution $(A, \phi)$ in temporal gauge. We now impose the Cronstr\"om gauge in an open set $S_p$ containing the past lightcone $K(p)$ from $p$ to the initial surface $\Sigma$, as described above. Note that the gauge transformation taking the temporal gauge solution $(A,\phi)$ to the Cronstr\"om gauge has $U(p) = \mathbbm{1}$, so it follows that $F(p)$, $\phi(p)$, and $(\D \phi)(p)$ are invariant under the gauge transformation. Using the above observation, we express the solutions to the wave equations \eqref{waveeqnF} and \eqref{waveeqnDphi} at $p$ as integrals of the nonlinearities (in Cronstr\"om gauge) over the past lightcone $K(p)$ of $p$ up to the initial surface $\Sigma$. Translating the point $p = (t_0,x_0)$ to the origin for convenience, the initial surface ends up at $\Sigma = \{ t = - t_0 \}$, and we find 
\begin{align} \label{lightconeintegralF}
	\begin{split} F_{\mu \nu}(0) & = \textcolor{olive}{F^{(0)}_{\mu \nu}(0)} \\
	& + \frac{1}{4\pi} \int_{K(t_0)} r \, \d r \, \d \Omega \, \Big\{ -2 \textcolor{olive}{ \nabla^c ( [A_c, F_{\mu \nu}])} + \textcolor{blue}{[\nabla_c A^c, F_{\mu \nu}]} - \textcolor{blue}{[A^c, [A_c, F_{\mu \nu}]]} \\
	&  + \textcolor{purple}{\left( (\D_\nu \phi) \cdot \theta_\alpha (\D_\mu \phi) - (\D_\mu \phi) \cdot \theta_\alpha (\D_\nu \phi) \right) \theta_\alpha} - 2 \textcolor{purple}{[F_\nu^{\phantom{\nu}c}, F_{\mu c}]} + \textcolor{orange}{\left( (F_{\mu \nu} \phi) \cdot \theta_\alpha \phi \right) \theta_\alpha} \Big\}(-r,r,\omega)
	\end{split}
\end{align}
and
\begin{align} \label{lightconeintegralDphi}
	\begin{split} (\D_\nu \phi)(0) & = \textcolor{olive}{(\D_\nu \phi)^{(0)}(0)} \\
	& + \frac{1}{4 \pi} \int_{K(t_0)} r \, \d r \, \d \Omega \, \Big\{ -2\textcolor{olive}{ \nabla^c ( A_c \D_\nu \phi)} + \textcolor{blue}{(\nabla^c A_c ) \D_\nu \phi} - \textcolor{blue}{A^c A_c \D_\nu \phi} - 2 \textcolor{purple}{F_\nu^{\phantom{\nu}c} \D_c \phi} \\
	& + \textcolor{orange}{\left( (\D_\nu \phi) \cdot \theta_\alpha \phi\right) \theta_\alpha \phi} - \textcolor{orange}{\lambda \D_\nu ( |\phi|^2 \phi)} \Big\}(-r,r,\omega),
	\end{split}
\end{align}
where the indices $\mu,\nu$ indicate contraction with the basis vectors $\partial/\partial x^\mu$, $\partial/\partial x^\nu$, so that $F_{\mu \nu}$ transforms as a scalar.

\begin{lemma} \label{lem:localizedestimates} The $L^\infty$ estimates of Eardley and Moncrief can be localized entirely to the lightcone. Specifically, one has the estimate
\[ N(t) \leq p(t) + q(t) \int_0^t N(s) \, \d s, \]
where
\[ N(s) = \| F(s) \|^2_{L^\infty(B(t-s))} + \| \D \phi(s) \|^2_{L^\infty(B(t-s))}, \]
and $p(t)$ and $q(t)$ are positive polynomials (perhaps containing positive fractional powers) in $t$, with coefficients depending on the $(H^2(B(t)) \times H^1(B(t)))^2$ norms of the temporal gauge initial data, the local energy $E_{\mathrm{loc}}$ in the lightcone from $p$ to $\Sigma$, and the $L^2$ norm of $\phi$ on $B(t) \cap \Sigma$.
\end{lemma}

\begin{proof} The terms on the right-hand sides of \eqref{lightconeintegralF} and \eqref{lightconeintegralDphi} are categorized by colour according to the types of techniques, due to Eardley and Moncrief \cite{EardleyMoncrief1982b}, required to estimate them. The olive-coloured terms in each equation (the linear part of the solution and the first term inside the integral) can be expressed explicitly in terms of the initial data; the blue terms (the second and third terms in each integral) are dealt with by using the Cronstr\"om gauge expressions \eqref{Cronstrompotential} and \eqref{CronstromdivA}; the purple terms (the fourth and fifth terms in the integral for $F$ and the fourth term in the integral for $\D \phi$) may be estimated by observing that they all contain exactly one factor encoding the flux across the lightcone; finally, the orange terms (the last term in the integral for $F$ and the last two terms in the integral for $\D \phi$) are estimated by relatively simple applications of the H\"older inequality and the Sobolev embedding theorems.

We briefly show how to localize one term from each colour class. No new techniques are required, and we refer the reader interested in the original derivation of the estimates to \cite{EardleyMoncrief1982b}. The olive terms 
\[ \textcolor{olive}{I^\text{olive}_1} \defeq F^{(0)}_{\mu \nu}(0) - \frac{1}{2 \pi} \int_{K(t_0)} \nabla^c ( [A_c, F_{\mu \nu}] ) \, r \, \d r \, \d \Omega \]
may be expressed explicitly, using the method of spherical means for the first term and by integrating by parts and using the condition $x^a A_a = 0$ for the second term, in terms of the temporal gauge initial data on the $2$-sphere defined by $\Sigma \cap K(t_0)$. Likewise for the terms
\[ \textcolor{olive}{I^\text{olive}_2} \defeq (\D_\nu \phi)^{(0)}(0) -\frac{1}{2 \pi} \int_{K(t_0)} \nabla^c(A_c \D_\nu \phi) \, r \, \d r \, \d \Omega. \]
The details are contained in equation $(2.39)$ of \cite{EardleyMoncrief1982b}.

For the blue terms, let us consider 
\[ \textcolor{blue}{I^\text{blue}} \defeq \int_{K(t_0)} (\nabla^c A_c) (\D_\nu \phi) \, r \, \d r \, \d \Omega. \]
Using the Cronstr\"om gauge expression \eqref{CronstromdivA} and the fact that $x^a F_{ab} = r l^a F_{ab} = r F_{lb}$ for $x \in K$, we find
\begin{align*} \textcolor{blue}{I^\text{blue}} & = \int_0^{r_0} \d r \int_{\mathbb{S}^2} \d \Omega \, r \int_0^1 \d s \Big\{ s^2 \left.[x^a F_{ab}(sx), A^b(sx)]\right|_K \\
& - s^2	 \left( x^a (\D_a \phi)(x)|_K \cdot (\theta_\alpha \phi)(sx)|_K \theta_\alpha \right) \Big\} ( \D_\nu \phi)(x)|_K \\
& = \int_0^{r_0} \d r \int_{\mathbb{S}^2} \d \Omega \, r \int_0^1 \d s \, s^2 r \left. [ F_{lb}(sx), A^b(sx)]\right|_K (\D_\nu \phi)(x)|_K \\
& - \int_0^{r_0} \d r \int_{\mathbb{S}^2} \d \Omega \, r \int_0^1 \d s \, s^2 r \left( (\D_l \phi)(sx)|_K \cdot (\theta_\alpha \phi)(sx)|_K \theta_\alpha \right) (\D_\nu \phi)(x)|_K \\
& \eqdef \textcolor{blue}{I^\text{blue}_1} - \textcolor{blue}{I^\text{blue}_2}.
\end{align*}
Consider the above summands separately. Using \eqref{Cronstrompotential} and making the change of variables $(sr, ur) = (r', \bar{r})$, for the first one we have
\begin{align*} \textcolor{blue}{I^\text{blue}_1} & = \int_0^{r_0} \d r \int_{\mathbb{S}^2} \d \Omega \, r^3 \int_0^1 \d s \, s^2 \int_0^1 \d u \, u  \left.[ F_{lb}(sx), F_l^{\phantom{l}b}(ux) ]\right|_K (\D_\nu \phi)(x)|_K \\
& = \int_0^{r_0} \d r \int_{\mathbb{S}^2} \d \Omega \int_0^1 \d s \int_0^1 \d u \, r^3 s^2 u [ F_{lA}(-sr,sr,\omega), F_{lA}(-ur,ur,\omega) ] (\D_\nu \phi)(-r,r,\omega) \\
& = \int_0^{r_0} \d r \frac{1}{r^2} \int_{\mathbb{S}^2} \d \Omega \int_0^r \d r' \int_0^r \d \bar{r} \, (r')^2 \bar{r} [ F_{lA}(-r',r',\omega), F_{lA}(-\bar{r}, \bar{r},\omega) ] (\D_\nu \phi)(-r,r,\omega) \\
& \leq C \int_0^{r_0} \d r \frac{1}{r^2} \int_{\mathbb{S}^2} \d \Omega \int_0^r \d r' \int_0^r \d \bar{r} \, (r')^2 \bar{r} |F_{lA}(-r',r',\omega) | |F_{lA} ( - \bar{r}, \bar{r}, \omega ) \| \D \phi (-r) \|_{L^\infty(B(r))} \\
& \leq C \int_0^{r_0} \d r \frac{1}{r} \int_{\mathbb{S}^2} \d \Omega \left( \int_0^r \d r' \, r' | F_{lA}(-r',r',\omega) | \right)^2 \| \D \phi (-r) \|_{L^\infty(B(r))} \\
& \leq C \int_0^{r_0} \d r \, \| F_{lA} \|^2_{L^2(K(r))} \| \D \phi(-r) \|_{L^\infty(B(r))},
\end{align*}
where $|F|$ denotes the Frobenius norm of $F$, $K(r)$ is the subcone of $K(t_0)$ of height $r$, and we have used the Cauchy--Schwarz inequality in the last line. Using the energy identity \eqref{energyidentitycone}, we thus have the estimate
\[ \textcolor{blue}{I^\text{blue}_1} \leq C E_{\mathrm{loc}} \int_0^{t_0} \| \D \phi(-t) \|_{L^\infty(B(t))} \, \d t. \]
To estimate $\textcolor{blue}{I^\text{blue}_2}$, we make the same change of variables $sr = r'$ to get
\begin{align*} \textcolor{blue}{I^\text{blue}_2} & = \int_0^{r_0} \d r \int_{\mathbb{S}^2} \d \Omega \int_0^r \d r' \frac{1}{r} (r')^2 \left( (\D_l \phi)(-r',r',\omega) \cdot (\theta_\alpha \phi)(-r',r',\omega) \right) (\theta_\alpha \D_\nu \phi)(-r,r,\omega) \\
& \leq C \int_0^{r_0} \d r \, \| \D \phi(-r) \|_{L^\infty(B(r))}	 \frac{1}{r} \int_{\mathbb{S}^2} \d \Omega \int_0^r \d r' (r')^2 | \D_l \phi |(-r',r',\omega) |\phi|(-r',r',\omega).
\end{align*}
Using H\"older's inequality with exponents $(3,2,6)$, one has
\begin{align*} \textcolor{blue}{I^\text{blue}_2} & \leq C \int_0^{r_0} \d r \, \| \D \phi (-r) \|_{L^\infty(B(r))} \frac{1}{r} \left( \int_0^r (r')^2 \, \d r' \right)^{1/3} \left(	 \int_{\mathbb{S}^2} \d \Omega \int_0^r \d r' \, (r')^2 | \D_l \phi|^2(-r',r',\omega) \right)^{1/2} \\
& \times \left( \int_{\mathbb{S}^2} \d \Omega \int_0^r \d r' \, (r')^2 |\phi|^6(-r',r',\omega) \right)^{1/6} \\
& \leq C \int_0^{r_0} \d r \, \| \D \phi(-r) \|_{L^\infty(B(r))} \| \phi \|_{L^6(K(r))} \| \D_l \phi \|_{L^2(K(r))}.
\end{align*}
Now $\| \D_l \phi \|_{L^2(K(r))} \leq C E^{1/2}_{\mathrm{loc}}$ is immediate by \eqref{energyidentitycone}, and since
\[ \frac{\d}{\d r'} ( \phi(-r',r',\omega) ) = (l^a \nabla_a \phi)(-r',r',\omega), \]
by the gauge-invariant Sobolev estimate of Jaffe--Taubes (see \S6 of \cite{JaffeTaubes1980}) one has
\[ \| \phi \|_{L^6(K(r))} \leq C \left( \| \D^\| \phi \|_{L^2(K(r))} + \| \phi \|_{L^2(K(r))} \right), \]
where $\D^\|  = (\D_l, \D_A)$. We show in the appendix that the $L^2$ norm of $\phi$ on the cone can be controlled by the local energy and the $L^2$ norm of $\phi$ at the base of the cone, $\| \phi \|_{L^2(K(r))} \leq 2 E^{1/2}_{\mathrm{loc}} t_0 + \| \phi \|_{L^2(B(r_0))}$. We thus conclude that 
\[ \textcolor{blue}{I^\text{blue}_2} \leq C E^{1/2}_{\mathrm{loc}} \left( 2 t_0 E^{1/2}_{\mathrm{loc}} + \| \phi \|_{L^2(B(r_0))} \right) \int_0^{t_0} \| \D \phi (-t) \|_{L^\infty(B(t))} \, \d t. \]

For the purple terms, we consider as an example the term
\[ \textcolor{purple}{I^\text{purple}} \defeq \int_{K(t_0)} (F_\nu^{\phantom{\nu}c} \D_c \phi) \, r \, \d r \, \d \Omega. \]
Expanding the product, we have
\[ F_\nu^{\phantom{\nu}c} \D_c \phi = - \frac{1}{2} F_{\nu l} \D_n \phi - \frac{1}{2} F_{\nu n} \D_l \phi + F_{\nu A} \D_A \phi, \]
so the last two terms can be estimated by
\begin{align*} & \int_0^{r_0} \d r \int_{\mathbb{S}^2} \d \Omega \, \| F(-r) \|_{L^\infty(B(r))} r | \D^\| \phi |(-r,r,\omega) \\ 
& \leq C \left( \int_0^{r_0} \d r \, \| F(-r) \|^2_{L^\infty(B(r))} \right)^{1/2} \left( \int_0^{r_0} \d r \int_{\mathbb{S}^2} \d \Omega \, r^2 | \D^\| \phi |^2(-r,r,\omega) \right)^{1/2} \\
& \leq C E^{1/2}_{\mathrm{loc}} \left( \int_0^{t_0} \| F(-t) \|^2_{L^\infty(B(t))} \, \d t \right)^{1/2}.
\end{align*}
To estimate the first term, we introduce the basis consisting of $e_0 = \partial_t$, $e_1 = \partial_r$, and $e_A$. One then has
\[ e_0 = \frac{1}{2} (n - l) \quad \text{and} \quad e_1 = \frac{1}{2} (n + l), \]
and that the Cartesian basis $\partial/\partial x^j$ for $\mathbb{R}^3$ is related to the basis $\{ e_1, e_A \}$ by an orthogonal transformation $O$,
\[ \frac{\partial}{\partial x^j} = O_{jk} e_k, \qquad e_k = O_{jk} \frac{\partial}{\partial x^j}. \]
If $\nu = t$, using $\partial_t = \frac{1}{2}(n-l)$ the first term then reads
\[ F_{t l} \D_n \phi = \frac{1}{2} F_{nl} \D_n \phi.  \]
One can thus estimate
\begin{align*} \int_{K(t_0)} r \, \d r \, \d \Omega \, |F_{tl} \D_n \phi | &\leq C \int_0^{r_0} \d r \int_{\mathbb{S}^2} \d \Omega \, \| \D \phi(-r) \|_{L^\infty(B(r))} r | F_{nl} |(-r,r,\omega) \\ 
& \leq C E^{1/2}_{\mathrm{loc}} \left( \int_0^{t_0} \| \D \phi(-t) \|^2_{L^\infty(B(t))} \, \d t \right)^{1/2}. 
\end{align*}
If, on the other hand, $\nu = i$, then
\[ F_{il} = O_{im} F_{e_m l} = O_{i1} F_{e_1 l} + O_{iA} F_{Al} = O_{i1} \frac{1}{2} F_{nl} + O_{iA} F_{Al},  \]
so a similar estimate can be deduced.

Finally, for the orange terms let us consider as an example the term
\[ \textcolor{orange}{I^\text{orange}} \defeq \int_{K(t_0)} \left( (\D_\nu \phi) \cdot \theta_\alpha \phi \right) (\theta_\alpha \phi) \, r \, \d r \, \d \Omega. \]
Applying Cauchy--Schwarz, we have
\begin{align*} \textcolor{orange}{I^\text{orange}} & \leq C \int_0^{r_0} \d r \int_{\mathbb{S}^2} \d \Omega \, r \, \| \D \phi (-r) \|_{L^\infty(B(r))} | \phi |^2(-r,r,\omega) \\
& \leq C \left( \int_0^{r_0} \d r \int_{\mathbb{S}^2} \d \Omega \, r^2 |\phi|^4(-r,r,\omega) \right)^{1/2} \left( \int_0^{r_0} \d r \, \| \D \phi (-r) \|^2_{L^\infty(B(r))} \right)^{1/2} \\
& \leq C \| \phi \|^2_{L^4(K(t_0))} \left( \int_0^{r_0} \| \D \phi(-r) \|^2_{L^\infty(B(r))} \right)^{1/2}.	
\end{align*}
By Gagliardo--Nirenberg interpolation and the Jaffe--Taubes invariance argument, we have
\[ \| \phi \|_{L^4(K(t_0))} \leq C \left( \| \D^\| \phi \|^{3/4}_{L^2(K(t_0))} \| \phi \|^{1/4}_{L^2(K(t_0))} + \| \phi \|_{L^2(K(t_0))} \right) \]
for some constant $C>0$, so it follows that $\| \phi \|^2_{L^4(K(t_0))}$ can be estimated by a polynomial (containing perhaps fractional positive powers) in $E_{\mathrm{loc}}$, $t_0$, and the $L^2$ norm of $\phi$ on the base of the cone $K(t_0)$.

Going back to \eqref{waveeqnF} and \eqref{waveeqnDphi}, altogether the above estimates imply the bounds
\begin{align*} & \| F(0) \|^2_{L^\infty(B(0))} \leq p_1(t_0) + q_1(t_0) \int_0^{t_0} \left( \| \D \phi(-t) \|^2_{L^\infty(B(t))} + \| F(-t) \|^2_{L^\infty(B(t))} \right) \d t, \\
& 	\| \D \phi(0) \|^2_{L^\infty(B(0))} \leq p_2(t_0) + q_2(t_0) \int_0^{t_0} \left( \| \D \phi(-t) \|^2_{L^\infty(B(t))} + \| F(-t) \|^2_{L^\infty(B(t))} \right) \d t, 
\end{align*}
where $p_{1,2}(t_0)$, $q_{1,2}(t_0)$ are positive polynomials in $t_0$ with coefficients depending only on $E_{\mathrm{loc}}$ and the temporal gauge initial data (including $\| \phi \|_{L^2(B(r_0))}(-t_0)$) on $\Sigma \cap \mathbf{K}(t_0)$. Translating the origin so that $p$ has coordinates $(t, 0)$, the lemma follows.
\end{proof}

Given the result of \Cref{lem:localizedestimates}, one now wishes to apply Gr\"onwall's lemma to deduce that the uniform norm $N$ does not blow up. Some care is required at this point, since the function $N(s)$ may not be continuous in $s$. Indeed, continuity may fail in the second variable of the function
\[ f(s_1, s_2) = \| F (s_1) \|_{L^\infty(B(t-s_2))} \]
if one considers a function $F(s_1)$ with multiple maxima in $\overline{B(t)}$. But to apply Gr\"onwall's lemma one only needs to show  that $ | N(s) | \, \d s$ defines a locally finite measure,
\[ \int_0^t | N(s) | \, \d s < \infty. \]
But this is clear, since by Sobolev embedding
\begin{align*} \| F(s) \|^2_{L^\infty(B(t-s))} + \| \D \phi(s) \|^2_{L^\infty(B(t-s))}  & \leq \| F(s) \|^2_{L^\infty(B(t))} + \| \D \phi(s) \|^2_{L^\infty(B(t))} \\
& \la \| F(s) \|^2_{H^2(B(t))} + \| \D \phi(s) \|^2_{H^2(B(t))} \\
& \in L^\infty_{\mathrm{loc}}(\mathbb{R}_s),
\end{align*}
where the last inclusion follows from the results of Goganov--Kapitanskii \cite{GoganovKapitanskii1985}, see e.g. Thereom 3 therein. We thus obtain
\[ N(t) < \infty \quad \forall t > 0. \]
The construction can be repeated for any point $p \in \mathbb{M}$, so we can package the above work into the following theorem.

\begin{theorem} \label{thm:localuniformestimates} Consider temporal gauge initial data $(\mathbf{A}, \mathbf{E}, \phi, \pi) \in ( H^2_{\mathrm{loc}}(\mathbb{R}^3) \times H^1_{\mathrm{loc}}(\mathbb{R}^3) )^2$ for the system \eqref{Minkowskifieldequations} satisfying the constraint \eqref{Minkowskiconstraint}. Then the fields $F$ and $\D \phi$ are $L^\infty_{\mathrm{loc}}(\mathbb{R} \times \mathbb{R}^3)$ in the domain of existence of the solution.
\end{theorem}

\section{Gluing onto the Einstein Cylinder} \label{sec:gluing}

In this section we explain how the local uniform estimates on Minkowski space can be used to deduce global uniform estimates on the Einstein cylinder. It pays to state clearly what we shall be doing: we will prescribe initial data on the Einstein cylinder $\mathfrak{E}$, and consider a copy of Minkowski space $\mathbb{M}$ conformally embedded in $\mathfrak{E}$ in such a way that the initial surface in $\mathfrak{E}$ coincides with the initial surface in $\mathbb{M}$, as depicted in \cref{fig:Minkowskiembedding} below. Initial data on $\mathfrak{E}$ prescribed in this way will define initial data for the system on $\mathbb{M}$, however, because it will generically be non-zero all around the $3$-sphere, the corresponding data on $\mathbb{M}$ will have infinite energy. Nonetheless, it will be locally $(H^2 \times H^1)^2$, allowing us to deduce local $L^\infty$ estimates in $\mathbb{M}$ as per the previous chapter. We shall then transport these local estimates back to $\mathfrak{E}$, and patch them all the way around the $3$-sphere.

It is classical that Minkowski space $(\mathbb{M}, \eta = \d t^2 - \d r^2 - r^2 \mathfrak{s}_2)$ can be conformally embedded into the Einstein cylinder $(\mathfrak{E}, \mathfrak{e} = \d \tau^2 - \mathfrak{s}_3)$ using the conformal factor
\[ \Omega = 2 \cos ( \arctan(t-r) ) \cos (\arctan(t+r) ) = \frac{2}{ \sqrt{1+(t-r)^2} \sqrt{1+(t+r)^2} }.  \]
One has $\Omega^2 \eta = \mathfrak{e} = \d \tau^2 - \d \zeta^2 - (\sin^2 \zeta) \mathfrak{s}_2$, where the coordinates on the Einstein cylinder are related to the coordinates on Minkowski space by $\tau = \arctan(t-r) + \arctan(t+r)$, $\zeta = \arctan(t+r)- \arctan(t-r)$, and $\mathbb{M}$ is the subset of $\mathfrak{E} = \mathbb{R}_\tau \times \mathbb{S}^3$ given by 
\[ \mathbb{M} = \{ (\tau, \zeta) \, : \, |\tau| + \zeta < \pi, \, \zeta \geq 0 \} \times \mathbb{S}^2. \]
A picture of this embedding (for $t \geq 0$) is shown below.
\begin{figure}[thp]
	\centering
		\begin{minipage}{0.45\textwidth}
			\begin{tikzpicture}	
				\centering \hspace{-1.5cm}
				\node[inner sep=0pt] (embedding) at (4,0) 
					{\includegraphics[width=0.43\textwidth]{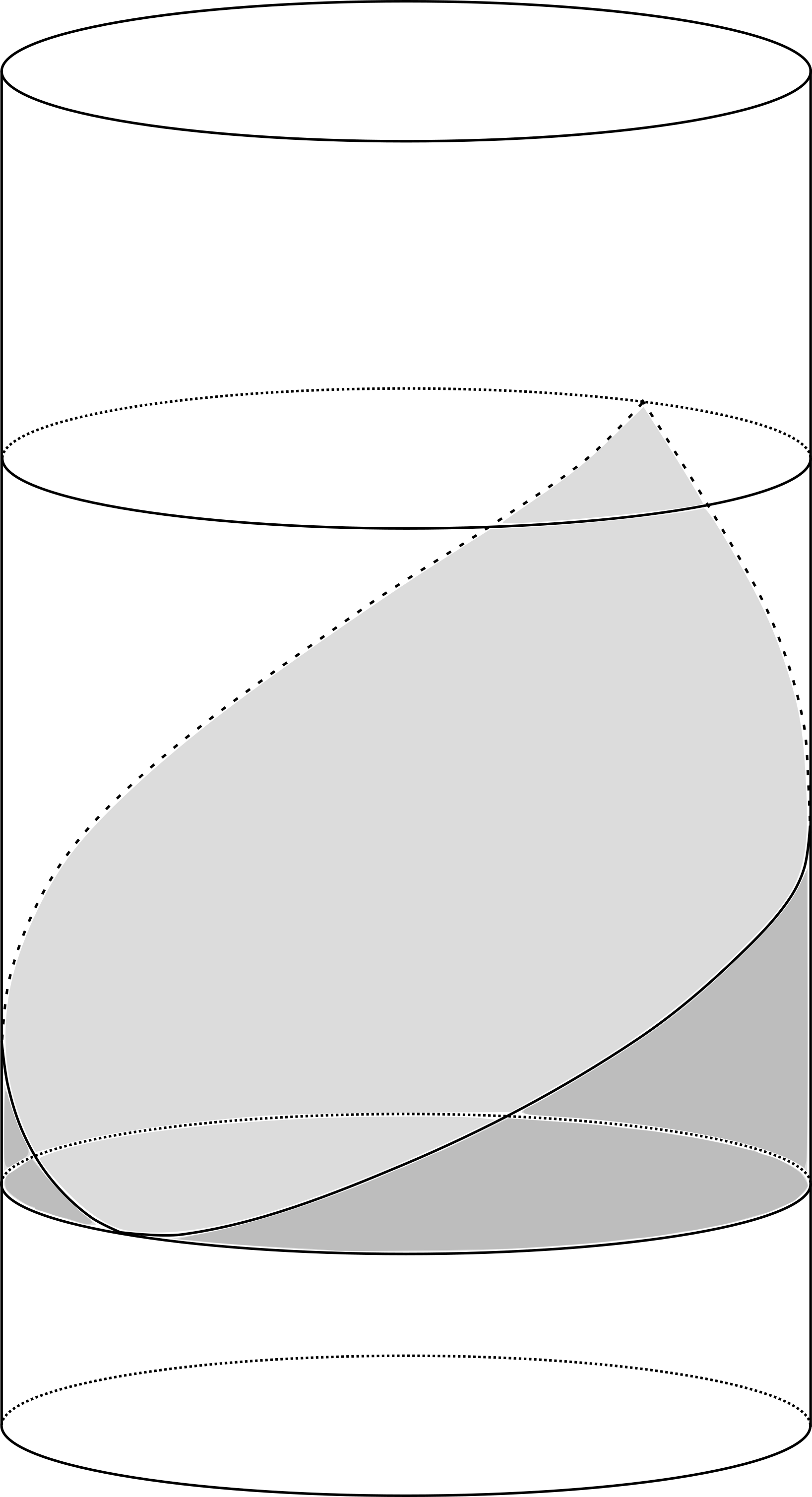}};
						\node[label={[shift={(6.5,-0.7)}]$\tau$}] {};
						\draw[->] (6.2,-1.55) .. controls (6.2,0) .. (6.2,.75);
			\end{tikzpicture}
		\end{minipage}
			\begin{minipage}{0.45\textwidth}
				\begin{tikzpicture}	
					\centering \hspace{-1.5cm}
					\node[inner sep=0pt] (embedding) at (4,0) 
						{\includegraphics[width=1.2\textwidth]{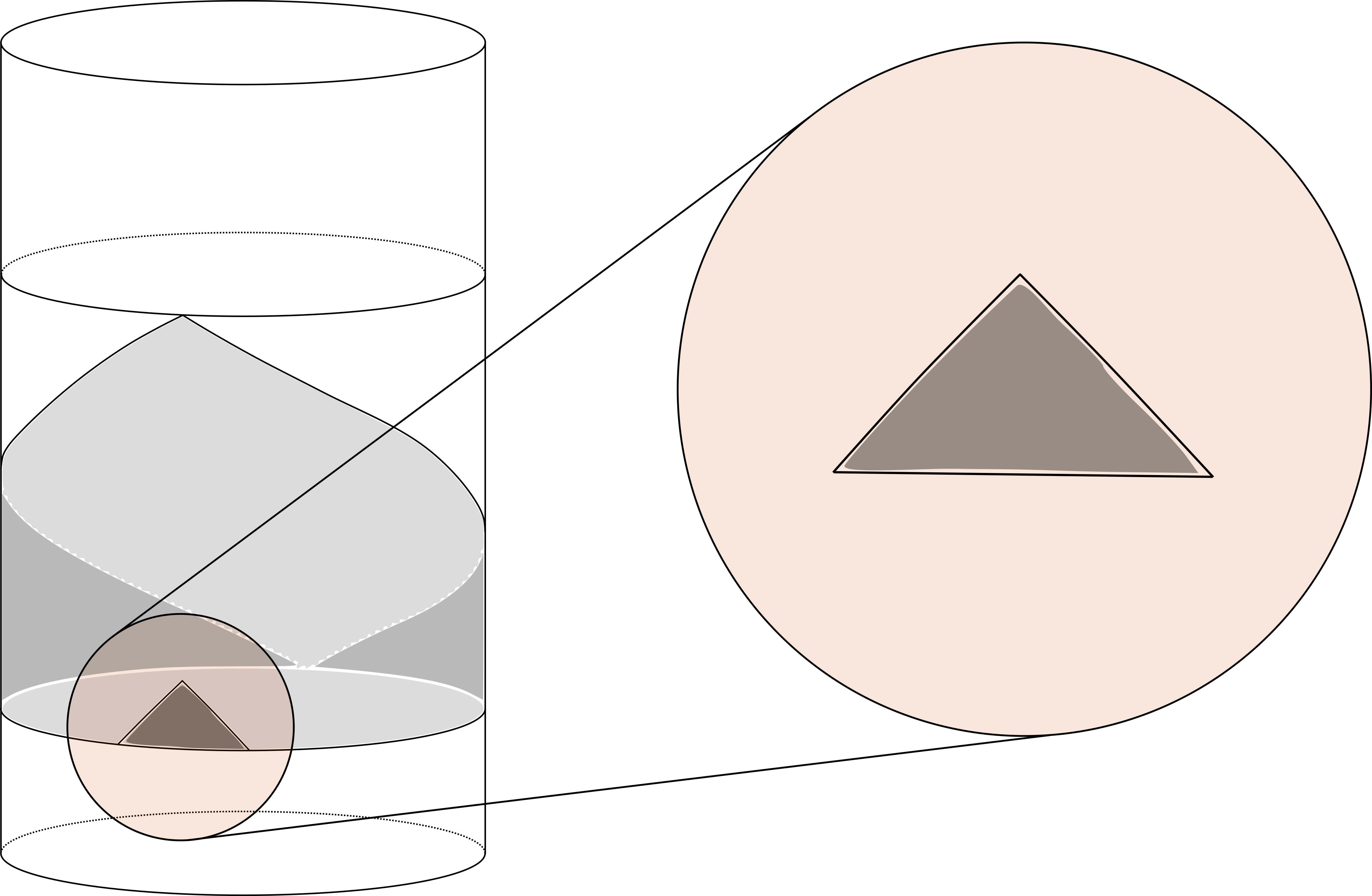}};
							\node[label={[shift={(6.15,0)}]$\hat{K}$}] {};
							\node[label={[shift={(6.15,-0.8)}]$O$}] {};
				\end{tikzpicture}
			\end{minipage}
	\label{fig:Minkowskiembedding}
	\caption{The embedding of $\mathbb{M}$ into $\mathfrak{E}$.}
\end{figure}
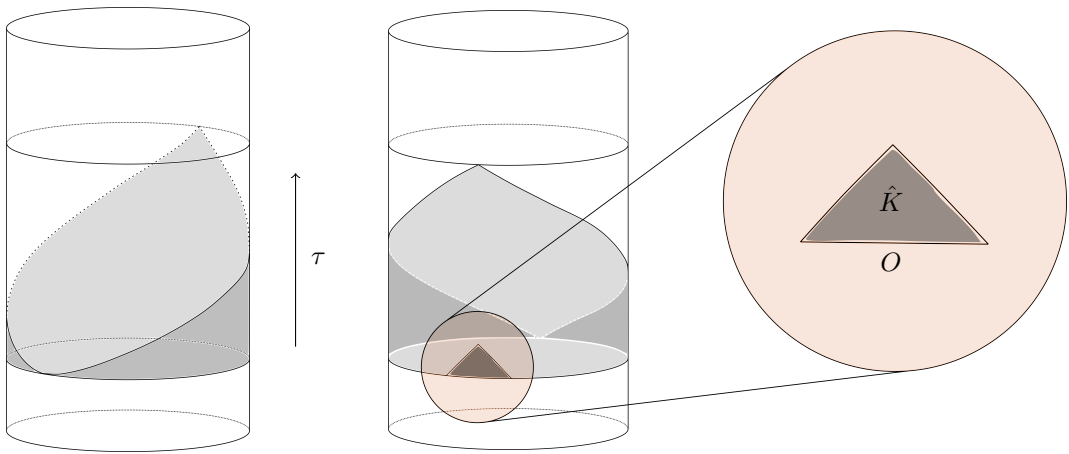

\noindent Instead of considering the whole of $\mathbb{M}$ embedded into $\mathfrak{E}$, we only consider the domain of dependence of a small ball in $\mathbb{M}$ glued onto $\mathfrak{E}$. Let $B(r_0)$ be the ball of radius $r_0$ centred at the origin $O \in \mathbb{M}$, and consider the cone $K = D^+(B(r_0))$. We consider the image $\hat{K}$ of $K$ under the embedding $\mathbb{M} \hookrightarrow \mathfrak{E}$; as conformal transformations preserve the causal structure, $\hat{K}$ is the domain of dependence of $\widehat{B(r_0)}$, where $\widehat{B(r_0)}$ is the image of $B(r_0)$ under the embedding.

\subsection{Conformal Transport of Estimates}

As already mentioned, it is classical that the weights 
\begin{equation} \label{conformalweights} A_a = \hat{A}_a \quad \text{and} \quad \phi = \Omega \hat{\phi}	
\end{equation}
leave the system \eqref{fieldequations} invariant under the conformal transformation $g_{ab} \leadsto \hat{g}_{ab} = \Omega^2 g_{ab}$. As a result, the fields $F_{ab}$ and $\D_a \phi$ transform according to $F_{ab} = \hat{F}_{ab}$ and $\hat{\D}_a \hat{\phi} = \Omega^{-1} ( \D_a \phi - \Upsilon_a \phi)$, where $\Upsilon_a = \partial_a \log \Omega$. Consider a cone $K$ with image $\hat{K}$ under the embedding $\mathbb{M} \hookrightarrow \mathfrak{E}$, as described above. It is clear that $ 0 < C_1 \leq | \Omega | \leq C_2 < \infty$ in $K$, so immediately $ \| \phi \|_{L^\infty(K)} \simeq \| \hat{\phi} \|_{L^\infty(\hat{K})}$. Indeed, for example
\[ | \Omega^{-1} | \leq \frac{1}{2} \left| \sqrt{1+ (t-r)^2} \sqrt{1+(t+r)^2} \right| \leq \frac{1}{2} ( 1 + 4 r_0^2), \]
and \[ | \Omega | \leq 2. \]
To deduce the same type of equivalence for tensor fields, one needs to check that the norms defined by the Riemannian metrics
\[ \Gamma^{ab} = 2 T^a T^b - \eta_{ab} \quad \text{and} \quad \hat{\Gamma}^{ab} = 2 \hat{T}^a \hat{T}^b - \mathfrak{e}^{ab}, \]
where $T^a = \partial_t$ and $\hat{T}^a = \partial_\tau$, are equivalent, at least in $K$.

\begin{proposition} For any $1$-form $X_a$ one has $| X |_{\Gamma} \simeq | X |_{\hat{\Gamma}}$ in $K$.	
\end{proposition}
\begin{proof} By a direct calculation using the chain rule, one finds
\[ T^a = \frac{1}{4} \Omega^2 \left( (2 + u^2 + v^2) \hat{T}^a + (u^2 - v^2) \hat{Z}^a \right), \]
where $\hat{Z}^a = \partial_\zeta$, $u = t- r$, and $v=t+r$. A further calculation then shows that
\begin{align*} \Omega^{-2} |X|^2_{\Gamma} &= \frac{1}{8} \Omega^2 \left( (2+u^2+v^2)^2 - 1 \right) (\hat{T}^a X_a)^2 + \frac{1}{4}  \Omega^2 ( 2 + u^2 + v^2)(u^2 - v^2) (\hat{T}^a X_a)(\hat{Z}^a X_a) \\
& + (u^2 - v^2)^2 (\hat{Z}^a X_a)^2 + |X|^2_{\mathfrak{s}_3}	.
\end{align*}
It is clear that $|X|^2_{\Gamma} \la |X|^2_{\hat{\Gamma}}$, while for the lower bound it is enough to observe that
\[ (2+u^2+v^2)(u^2 - v^2)(\hat{T}^a X_a)(\hat{Z}^a X_a) \geq - \frac{1}{4} ( 2+ u^2 +v^2)^2 (\hat{T}^a X_a)^2  - (u^2-v^2)^2 (\hat{Z}^a X_a)^2, \]
so that
\begin{align*} \Omega^{-2} | X |^2_{\Gamma} & \geq \frac{1}{8} \Omega^2 \left( \frac{1}{2} (2+u^2+v^2)^2 - 1 \right) (\hat{T}^a X_a)^2 \\
& + \left((u^2-v^2)^2 \left( 1 - \frac{\Omega^2}{4} \right) \right) (\hat{Z}^a X_a)^2 + |X|^2_{\mathfrak{s}_3} \\
& \geq \frac{1}{8} \Omega^2 ( \hat{T}^a X_a)^2 + |X|^2_{\mathfrak{s}_3} \\
& \geq \frac{1}{8} \Omega^2 | X |^2_{\hat{\Gamma}},
\end{align*}
as $\Omega^2/4 \leq 1$. \end{proof}

\noindent It follows that
\[ \| F \|_{L^\infty(K)} \simeq \| \hat{F} \|_{L^\infty(\hat{K})} \quad \text{and} \quad \| \hat{\D} \hat{\phi} \|_{L^\infty(\hat{K})} \la \| \D \phi \|_{L^\infty(K)} + \| \Upsilon \phi \|_{L^\infty(K)}. \]
Note that these are gauge-independent. This demonstrates that local $L^\infty$ estimates on Minkowski space imply local $L^\infty$ estimates on the Einstein cylinder. We show below how initial data on the Einstein cylinder defines initial data on Minkowski space, and use this to complete our construction.

Consider temporal gauge (with respect to $\partial_\tau$) initial data $(\hat{\mathbf{A}}, \hat{\mathbf{E}}, \hat{\phi}, \hat{\pi}) \in (H^2(\mathbb{S}^3) \times H^1(\mathbb{S}^3))^2$ for the Yang--Mills--Higgs equations on $\mathfrak{E}$,
\begin{equation} \label{fieldequationscylindertemporalgauge} \dot{\hat{\mathbf{E}}}_i + \slashgrad^j \hat{F}_{ij} + [ \hat{\mathbf{A}}^j, \hat{F}_{ij} ] = ( ( \hat{\boldsymbol{\D}}_i \hat{\phi} ) \cdot \theta_\alpha \hat{\phi} ) \theta_\alpha, \qquad \dot{\hat{\pi}} - \hat{\boldsymbol{\D}}^j \hat{\boldsymbol{\D}}_j \hat{\phi} + \hat{\phi} + \lambda |\hat{\phi}|^2 \hat{\phi} = 0,
\end{equation}
satisfying the constraint
\begin{equation} \label{cylinderconstraint} \slashgrad^j \hat{\mathbf{E}}_j + [ \hat{\mathbf{A}}^j, \hat{\mathbf{E}}_j ] = (\hat{\pi} \cdot \theta_\alpha \hat{\phi} ) \theta_\alpha. \end{equation}
Since $\hat{T}^a = \partial_\tau$ is not everywhere parallel to $T^a = \partial_t$, the temporal gauge on $\mathfrak{E}$ is of course not the same as the temporal gauge on $\mathbb{M}$. However, $\hat{T}^a$ and $T^a$ \emph{are} parallel on the initial surface $\Sigma_0 = \{ \tau = 0 \} = \{ t = 0 \}$,
\[ \left.r_+^2 \frac{\partial}{\partial t}\right|_{t=0} = \left. \frac{\partial}{\partial \tau}\right|_{\tau = 0}, \]
where $r_+^2 = (1+r^2)/2$. Thus on the initial surface $\Sigma_0$ one has $A_0 = 0 ~ \text{a.e.} \iff \hat{A}_0 = 0 ~ \text{a.e.}$. The data $(\hat{\mathbf{A}}, \hat{\mathbf{E}}) \in H^2(\mathbb{S}^3) \times H^1(\mathbb{S}^3)$ then gives rise to temporal gauge initial data $(\mathbf{A}, \mathbf{E}) \in H^2_{\mathrm{loc}}(\mathbb{R}^3) \times H^1_{\mathrm{loc}}(\mathbb{R}^3)$ on Minkowski space: one has
\[ \left. \hat{\mathbf{A}}_a \right|_{\tau = 0} = \mathbf{A}_a \Big|_{t=0} \quad \text{and} \quad \hat{\mathbf{E}}_a \Big|_{\tau = 0} = r_+^2 \mathbf{E}_a \Big|_{t=0}. \]
For the scalar field part, one similarly has
\[ \hat{\phi} \Big|_{\tau = 0} = r_+^2 \phi \Big|_{t=0}, \]
and (since $(\partial_t \Omega)|_{t=0} = 0$),
\[ (\partial_\tau \hat{\phi} ) \Big|_{\tau = 0} = (\Omega^{-1} \partial_t \hat{\phi} ) \Big|_{t=0} = (\Omega^{-2} \partial_t \phi) \Big|_{t=0} = r^4_+ (\partial_t \phi) \Big|_{t=0}, \]
i.e.
\[ \hat{\pi} \Big|_{\tau = 0} = r^4_+ \pi \Big|_{t=0}. \]
Thus $(\hat{\phi}, \hat{\pi}) \in H^2(\mathbb{S}^3) \times H^1(\mathbb{S}^3)$ similarly gives rise to temporal gauge initial data $(\phi, \pi) \in H^2_{\mathrm{loc}}(\mathbb{R}^3) \times H^1_{\mathrm{loc}}(\mathbb{R}^3)$. Furthermore, that the Minkowskian initial data satisfies the constraint equation \eqref{Minkowskiconstraint} as a consequence of the constraint equation \eqref{cylinderconstraint} on the Einstein cylinder follows from the conformal invariance of the field equations and the fact that $\partial_t$ and $\partial_\tau$ are parallel initially. In summary, $(H^2 \times H^1)^2$ temporal gauge initial data on $\mathfrak{E}$ gives rise to $(H^2_{\mathrm{loc}} \times H^1_{\mathrm{loc}})^2$ temporal gauge initial data on $\mathbb{M}$.

\begin{remark} \label{rmk:infiniteenergy} The locality is necessary. Indeed, the measures on $\{ t = 0\}$ and $\{ \tau = 0 \}$ are related by
\[ \dvol_{\mathfrak{s}_3} = r_+^{-6} \dvol_{\mathbb{R}^3}, \]
so the $L^2$ norms of the initial data scale as
\begin{align*} & \int_{\mathbb{S}^3} | \hat{\phi} |^2 \dvol_{\mathfrak{s}_3} = \int_{\mathbb{R}^3} \frac{1}{r^2_+} |\phi|^2 \, \dvol_{\mathbb{R}^3}, \qquad \int_{\mathbb{S}^3} | \hat{\mathbf{A}} |^2 \dvol_{\mathfrak{s}_3} = \int_{\mathbb{R}^3} \frac{1}{r^2_+} | \mathbf{A} |^2 \dvol_{\mathbb{R}^3}, \\
& \int_{\mathbb{S}^3} | \hat{\pi} |^2 \dvol_{\mathbb{S}^3} = \int_{\mathbb{R}^3} r^2_+ | \pi |^2 \dvol_{\mathbb{R}^3}, \qquad \int_{\mathbb{S}^3} | \hat{\mathbf{E}} |^2 \dvol_{\mathbb{S}^3} = \int_{\mathbb{R}^3} r^2_+ |\mathbf{E}|^2 \dvol_{\mathbb{R}^3},
\end{align*}
where $|\hat{\mathbf{A}}|^2$ and $|\hat{\mathbf{E}}|^2$ are computed with respect to the metric on $\mathbb{S}^3$, while $|\mathbf{A}|^2$ and $|\mathbf{E}|^2$ are computed with respect to the metric on $\mathbb{R}^3$ as appropriate. One sees that finite energy on $\mathfrak{E}$ does not imply finite energy on $\mathbb{R}^3$, and allows $|\mathbf{A}|, \phi \sim r^{-1}$ tails, for example.
\end{remark}

Consider any local solution $(\hat{A}_a, \hat{\phi})$ on $\mathfrak{E}$ with $(H^2(\mathbb{S}^3) \times H^1(\mathbb{S}^3))^2$ initial data. Then the conformally related fields $(A_a, \phi) = (\hat{A}_a, \Omega \hat{\phi})$ restricted to $\mathbb{M}$ are a solution to the Yang--Mills--Higgs equations on $\mathbb{M}$ with $(H^2_{\mathrm{loc}} \times H^1_{\mathrm{loc}})^2$ initial data, so by the local $L^\infty$ estimates of \Cref{sec:localization} satisfy 
\[ \| F \|_{L^\infty(K)} + \| \D \phi \|_{L^\infty(K)} < \infty. \]
To show that this implies 
\[ \| \hat{F} \|_{L^\infty(\hat{K})} + \| \hat{\D} \hat{\phi} \|_{L^\infty(\hat{K})} < \infty, \]
it only remains to check that $\| \Upsilon \phi \|_{L^\infty}$ is bounded in $K$. We have $\| \Upsilon \phi \|_{L^\infty(K)} \leq \| \Upsilon \|_{L^\infty(K)} \| \phi \|_{L^\infty(K)}$, and can estimate $\Upsilon_a = \partial_a \log \Omega$ easily by, for example,
\[ | \Upsilon_t | \leq \left| \frac{(t-r)}{1 + (t-r)^2} + \frac{(t+r)}{1+(t+r)^2} \right| \leq 2 t \leq 2 r_0, \] 
for the $\Upsilon_t$ component, and similarly for the $ \Upsilon_r $ component. To estimate $\| \phi \|_{L^\infty}$, we make use of the temporal gauge condition on $\mathbb{M}$,
\[ \phi(t) = \phi(0) + \int_0^t \pi(s) \, \d s, \]
so that 
\begin{align*} \| \phi \|_{L^\infty(K)} &\leq \| \phi(0) \|_{L^\infty(B(r_0))} + t \| \pi \|_{L^\infty(K)} \\
& \leq \| \phi(0) \|_{H^2(B(r_0))} + r_0 \| \pi \|_{L^\infty(K)} \\
& < \infty.
\end{align*}
Since the $\| \phi \|_{L^\infty}$ norm is gauge-independent, this does not present any issues with respect to gauge choice. Thus $\| \Upsilon \phi \|_{L^\infty(K)} < \infty$, and we have
\[ \| \hat{F} \|_{L^\infty(\hat{K})} + \| \hat{\D} \hat{\phi} \|_{L^\infty(\hat{K})} < \infty. \]
Since the position of the cone $\hat{K}$ on the Einstein cylinder was arbitrary (inasmuch as the position of the embedded copy of Minkowski space was arbitrary in $\mathfrak{E}$), we have proven the following.

\begin{theorem} \label{thm:Linftyestimatesstrip} For given temporal gauge initial data $(\hat{\mathbf{A}}, \hat{\mathbf{E}}, \hat{\phi}, \hat{\pi}) \in (H^2(\mathbb{S}^3) \times H^1(\mathbb{S}^3))^2$ for the system \eqref{fieldequationscylindertemporalgauge} satisfying the constraint \eqref{cylinderconstraint}, the fields $\hat{F}$ and $\hat{\D} \hat{\phi}$ are $L^\infty([0,\tau_0] \times \mathbb{S}^3)$ for some $\tau_0$ independent of the size of the initial data.
\end{theorem}

\section{Global Existence on the Einstein Cylinder} \label{sec:globalexistence}

\subsection{Local Existence \'a la Choquet-Bruhat and Christodoulou}

\begin{theorem}[Choquet-Bruhat and Christodoulou, 1981, \cite{ChoquetBruhatChristodoulou1981}] \label{thm:ChoquetBruhatChristodouloulocalexistence} Let $(\hat{\mathbf{a}}, \hat{\mathbf{e}}, \hat{\phi}_0, \hat{\phi}_1) \in (H^{s}(\mathbb{S}^3) \times H^{s-1}(\mathbb{S}^3))^2$ and $\hat{a}_0 \in H^s(\mathbb{S}^3)$, $s \geq 2$, be initial data for the Yang--Mills--Higgs equations
\begin{equation} \label{YMHcylinder} \hat{\D}^b \hat{F}_{ab} = - ((\hat{\D}_a \hat{\phi}) \cdot \theta_\alpha \hat{\phi}) \theta_\alpha, \qquad \hat{\D}^a \hat{\D}_a \hat{\phi} + \hat{\phi} + \lambda | \hat{\phi} |^2 \hat{\phi} = 0 \end{equation}
on $\mathfrak{E}$ satisfying the constraint
\begin{equation} \label{YMHconstraint} \slashgrad^j \hat{\mathbf{e}}_j + [ \hat{\mathbf{a}}^j, \hat{\mathbf{e}}_j ] = (\hat{\pi} \cdot \theta_\alpha \hat{\phi}_0 ) \theta_\alpha,	
\end{equation}
where $\hat{\pi} = \hat{\phi}_1 + \hat{a}_0 \hat{\phi}_0$. Then there exists $\epsilon > 0$ such that there exists a solution
\[ \hat{A}_a, \hat{\phi} \in E_s((-\epsilon,\epsilon) \times \mathbb{S}^3) \defeq \bigcap_{k=0}^s C^k((-\epsilon, \epsilon); H^{s-k}(\mathbb{S}^3)) \]
to \eqref{YMHcylinder} in Lorenz gauge $\hat{\nabla}_a \hat{A}^a = 0$, with
\[ \hat{\mathbf{A}}\Big|_{\tau = 0} = \hat{\mathbf{a}}, \quad \hat{A}_0 \Big|_{\tau = 0 } = \hat{a}_0, \quad \hat{\mathbf{E}} \Big|_{\tau=0} = \hat{\mathbf{e}}, \quad \hat{\phi} \Big|_{\tau=0} = \hat{\phi}_0, \quad \dot{\hat{\phi}} \Big|_{\tau = 0} = \hat{\phi}_1. \]
The largest such number $\epsilon$ depends continuously on the size $M$ of the data, where
\[ M = \| \hat{\phi}_0 \|_{H^s} + \| \hat{\mathbf{a}} \|_{H^s} + \| \hat{\phi}_1 \|_{H^{s-1}} + \| \hat{\mathbf{e}} \|_{H^{s-1}} + \| \hat{a}_0 \|_{H^s}, \]
and tends to infinity as $M$ tends to zero. Furthermore, the solution is unique\footnote{It is worth recalling here that we work with a \emph{compact} connected gauge group $\mathrm{G}$.} up to gauge transformations preserving the Lorenz gauge.
\end{theorem}

\begin{remark} The component $\hat{A}_0$ is non-dynamical and the $\hat{a}_0$ component of the initial data can in fact be chosen to be zero without restricting the class of solutions (\emph{c.f.} \S4 of \cite{ChoquetBruhatChristodoulou1981}).	
\end{remark}

\begin{corollary} Let $(\hat{\mathbf{a}}, \hat{\mathbf{e}}, \hat{\phi}_0, \hat{\phi}_1) \in (H^2(\mathbb{S}^3) \times H^1(\mathbb{S}^3))^2$ be temporal gauge initial data for the system \eqref{YMHcylinder} on $\mathfrak{E}$, satisfying the constraint \eqref{YMHconstraint}. Then there exists $\epsilon > 0$ such that there exists a solution $(\hat{A}_a, \hat{\phi}) \in E_2((-\epsilon,\epsilon) \times \mathbb{S}^3)^2$ to \eqref{YMHcylinder} in temporal gauge, with
\[ \hat{\mathbf{A}} \Big|_{\tau =0} = \hat{\mathbf{a}}, \quad \hat{\mathbf{E}} \Big|_{\tau = 0} = \hat{\mathbf{e}}, \quad \hat{\phi} \Big|_{\tau = 0} = \hat{\phi}_0, \quad \hat{\pi} \Big|_{\tau = 0} = \hat{\phi}_1. \]
The largest such number $\epsilon$ depends continuously on the size $M'$ of the data, where
\[ M' = \| \hat{\phi}_0 \|_{H^2} + \| \hat{\mathbf{a}} \|_{H^2} + \| \hat{\phi}_1 \|_{H^1} + \| \hat{\mathbf{e}}_1 \|_{H^1}, \]
and tends to infinity as $M'$ tends to zero. Furthermore, the solution is unique up to gauge transformations preserving the temporal gauge.
\end{corollary}

\begin{proof} This is immediate from \Cref{thm:ChoquetBruhatChristodouloulocalexistence}, if one can demonstrate that there exists a gauge transformation from the Lorenz gauge to the temporal gauge preserving the requisite regularity. A general gauge transformation $U$ of the system \eqref{YMHcylinder} takes
\[ \hat{A}_a \leadsto U \hat{A}_a U^{-1} + U \partial_a U^{-1},  \]
so to set $\hat{A}_0 = 0$ one needs to solve $U \hat{A}_0 U^{-1} + U \partial_\tau U^{-1} = 0$, or equivalently	
\[ \hat{A}_0 = U^{-1} \partial_\tau U. \]
Since $\mathrm{G}$ is a compact connected matrix Lie group, there exists $u \in \mathfrak{g}$ such that $U = \e^u$, so in terms of $u$ the above equation becomes $\partial_\tau u = \hat{A}_0$. This has the solution
\[ u(\tau) = u(0) + \int_0^\tau \hat{A}_0 (\sigma) \, \d \sigma, \]
so choosing $u(0) = 0$ (and $\hat{a}_0 = 0$) gives the required gauge transformation.
\end{proof}

\begin{remark} \label{rmk:energyblowup} It is implicit in \Cref{thm:ChoquetBruhatChristodouloulocalexistence} that if the largest time of existence is finite, $\epsilon_{\mathrm{max}} < \infty$, then
\[ \| \hat{\phi}(\tau) \|_{H^2(\mathbb{S}^3)} + \| \hat{\mathbf{A}}(\tau) \|_{H^2(\mathbb{S}^3)} + \| \hat{\mathbf{E}}(\tau) \|_{H^1(\mathbb{S}^3)} + \| \hat{\pi}(\tau) \|_{H^1(\mathbb{S}^3)} \longrightarrow \infty \]
as $\tau \to \epsilon_{\mathrm{max}}$. We shall show that the time of existence is in fact infinite by showing that the above norm does not blow up in finite time.	
\end{remark}

\subsection{Energy Estimates}

On the Einstein cylinder $\mathfrak{E}$ we may take the stress-energy tensor for the system \eqref{fieldequations} to be
\begin{equation} \label{modifiedstresstensor} \hat{\mathbf{\Theta}}_{ab} = - \langle \hat{F}_{ac}, \hat{F}_b^{\phantom{b}c} \rangle + \frac{1}{4} \mathfrak{e}_{ab} \langle \hat{F}_{cd}, \hat{F}^{cd} \rangle + (\hat{\D}_a \hat{\phi} ) \cdot (\hat{\D}_b \hat{\phi}) - \frac{1}{2} \mathfrak{e}_{ab} (\hat{\D}_c \hat{\phi}) \cdot (\hat{\D}^c \hat{\phi}) + \frac{1}{2} \mathfrak{e}_{ab} | \hat{\phi} |^2 + \frac{1}{4} \lambda \mathfrak{e}_{ab} |\hat{\phi} |^4.	
\end{equation}
This differs from the canonical stress-energy tensor \eqref{stresstensor} on $\mathfrak{E}$ by the term $\frac{1}{6} \hat{R}_{ab} |\hat{\phi}|^2$, but satisfies the exact conservation law
\[ \hat{\nabla}^a \hat{\mathbf{\Theta}}_{ab} = 0. \]
It thus defines a conserved energy on $\mathfrak{E}$,
\begin{align*} \hat{E}_0 &= \int_{\mathbb{S}^3} \hat{\mathbf{\Theta}}_{00} \dvol_{\mathfrak{s}_3} \\
& = \int_{\mathbb{S}^3} \hat{\mathbf{\Theta}}_{ab} (\partial_\tau)^a (\partial_\tau)^b  \dvol_{\mathfrak{s}_3} \\
& = \frac{1}{2} \int_{\mathbb{S}^3} \left( | \hat{\mathbf{E}} |^2 + | \hat{\mathbf{B}} |^2 + | \hat{\pi} |^2 + | \hat{\boldsymbol{\slashed{\D}}} \hat{\phi} |^2 + | \hat{\phi} |^2 + \frac{1}{2} \lambda | \hat{\phi} |^4 \right) \dvol_{\mathfrak{s}_3},	\\
\end{align*}
satisfying
\[ \frac{\d \hat{E}_0}{\d \tau}  = 0, \]
where $\hat{\boldsymbol{\slashed{\D}}} \hat{\phi}$ is the projection onto $\mathbb{S}^3$ of $\D_a \phi$. Let us also define the approximate energies
\begin{equation} \hat{\mathcal{E}}_1(\tau) \defeq \frac{1}{2} \int_{\mathbb{S}^3} \left( | \hat{\mathbf{E}}|^2 + | \slashgrad \hat{\mathbf{A}} |^2 + | \hat{\mathbf{A}} |^2 + | \hat{\pi}|^2 + | \slashgrad \hat{\phi} |^2 + | \hat{\phi} |^2 \right) \dvol_{\mathfrak{s}_3}
\end{equation}
and
\begin{equation} \hat{\mathcal{E}}_2(\tau) \defeq \frac{1}{2} \int_{\mathbb{S}^3} \left( | \slashgrad \hat{\mathbf{E}} |^2 + | \slashgrad^2 \hat{\mathbf{A}} |^2 + | \slashgrad \hat{\pi} |^2 + | \slashgrad^2 \hat{\phi} |^2 \right) \dvol_{\mathfrak{s}_3}.
\end{equation}
It is clear that $(\hat{\mathcal{E}}_1 + \hat{\mathcal{E}}_2)^{1/2}$ is equivalent to the $(H^2 \times H^1)^2$ norm of the solution (in temporal gauge) on $\{ \tau \} \times \mathbb{S}^3$. By differentiating $\hat{\mathcal{E}}_1$ in $\tau$, integrating by parts and using the equations \eqref{fieldequationscylindertemporalgauge} and \eqref{cylinderconstraint}, one arrives at the estimate
\begin{align} \label{energy1estimate} \begin{split} \left| \frac{\d \hat{\mathcal{E}}_1}{\d \tau} \right| &\leq C \left( 1 + \| \hat{F}(\tau) \|_{L^\infty} + \| \hat{\D} \hat{\phi} (\tau) \|_{L^\infty} \right) \hat{\mathcal{E}}_1 + \lambda \| \hat{\pi} (\tau) \|_{L^\infty} \| \hat{\phi} \|^3_{L^3} \\
& \leq C \left( 1+ \| \hat{F}(\tau) \|_{L^\infty} + \| \hat{\D} \hat{\phi} (\tau) \|_{L^\infty} + \lambda \| \hat{\D} \hat{\phi}(\tau) \|_{L^\infty} \| \hat{\phi}(\tau) \|_{L^\infty} \right) \hat{\mathcal{E}}_1
\end{split}
\end{align}
where the constant $C>0$ depends only on the structure group $\mathrm{G}$ and the geometry of $\mathbb{S}^3$. One similarly finds that
\begin{equation} \label{energy2estimate} \left| \frac{\d \hat{\mathcal{E}}_2}{\d \tau} \right| \leq C \left( 1 + \| \hat{F}(\tau) \|_{L^\infty} + \| \hat{\D} \hat{\phi}(\tau) \|_{L^\infty} + \| \hat{\phi}(\tau) \|_{L^\infty} + \| \hat{\mathbf{A}}(\tau) \|_{L^\infty} \right)^2 (\hat{\mathcal{E}}_1 + \hat{\mathcal{E}}_2 ).	
\end{equation}
Putting together \eqref{energy1estimate} and \eqref{energy2estimate}, it follows that
\begin{equation} \label{energy12estimate} \left| \frac{\d}{\d \tau} (\hat{\mathcal{E}}_1 + \hat{\mathcal{E}}_2 ) \right| \leq C \left( 1 + \| \hat{F}(\tau) \|_{L^\infty} + \| \hat{\D} \hat{\phi}(\tau) \|_{L^\infty} + \| \hat{\phi}(\tau) \|_{L^\infty} + \| \hat{\mathbf{A}}(\tau) \|_{L^\infty} \right)^2 (\hat{\mathcal{E}}_1 + \hat{\mathcal{E}}_2 ).
\end{equation}
To estimate $\| \hat{\phi} \|_{L^\infty}$ and $\| \hat{\mathbf{A}} \|_{L^\infty}$, notice that $\partial_\tau \hat{\phi} = \hat{\pi}$ and $\partial_\tau \hat{\mathbf{A}} = \hat{\mathbf{E}}$ imply
\[ \hat{\phi}(\tau) = \hat{\phi}(0) + \int_0^\tau \hat{\pi}(\sigma) \, \d \sigma \quad \text{and} \quad \hat{\mathbf{A}}(\tau) = \hat{\mathbf{A}}(0) + \int_0^\tau \hat{\mathbf{E}}(\sigma) \, \d \sigma,  \]
which give the estimates
\begin{equation} \label{phiLinftyestimate} \| \hat{\phi}(\tau) \|_{L^\infty} \leq C \| \hat{\phi}_0 \|_{H^2} + \int_0^\tau \| \hat{\pi}(\sigma) \|_{L^\infty} \, \d \sigma \quad \text{and} \quad \| \hat{\mathbf{A}}(\tau) \|_{L^\infty} \leq C \| \hat{\mathbf{a}} \|_{H^2} + \int_0^\tau \| \hat{\mathbf{E}}(\sigma) \|_{L^\infty} \, \d \sigma.
\end{equation}
We thus have the following.

\begin{theorem} \label{thm:largedataexistence} Let $(\hat{\mathbf{a}}, \hat{\mathbf{e}}, \hat{\phi}_0, \hat{\phi}_1) \in (H^2(\mathbb{S}^3) \times H^1(\mathbb{S}^3))^2$ be temporal gauge initial data for the system \eqref{YMHcylinder} on $\mathfrak{E}$ satisfying the constraint \eqref{YMHconstraint}. Then there exists a global solution $(\hat{A}_a, \hat{\phi}) \in E_2(\mathbb{R} \times \mathbb{S}^3)^2$ to \eqref{YMHcylinder} in temporal gauge with
\[ \hat{\mathbf{A}} \Big|_{\tau = 0 } = \hat{\mathbf{a}}, \quad \hat{\mathbf{E}} \Big|_{\tau = 0} = \hat{\mathbf{e}}, \quad \hat{\phi} \Big|_{\tau =0} = \hat{\phi}_0, \quad \text{and} \quad \hat{\pi} \Big|_{\tau = 0} = \hat{\phi}_1. \]
Furthermore, the solution is unique up to gauge transformations preserving the temporal gauge.
\end{theorem}

\begin{proof} Let $\epsilon_{\mathrm{max}} > 0$ be the maximal time of existence guaranteed by \Cref{thm:ChoquetBruhatChristodouloulocalexistence}. As per \Cref{rmk:energyblowup}, either $\epsilon_{\mathrm{max}} = \infty$ or the $(H^2 \times H^1)^2$ norm of the solution blows up as $\tau \to \epsilon_{\mathrm{max}}$. We show that the former is true by assuming that $\epsilon_{\mathrm{max}} < \infty$ and deriving a contradiction. We work with $\tau \geq 0$; the following argument applies equally well in the case $\tau < 0$. The local solution $(\hat{A}_a, \hat{\phi})$ satisfies
\[ \| \hat{\mathbf{A}} (\tau) \|_{H^2} + \| \hat{\mathbf{E}}(\tau) \|_{H^1} + \| \hat{\phi}(\tau) \|_{H^2} + \| \hat{\pi}(\tau) \|_{H^1} < \infty \]
for all $\tau < \epsilon_{\mathrm{max}}$, and in particular at $\tau = \epsilon_{\mathrm{max}} - \tau_0/2$, where $\tau_0$ is as in \Cref{thm:Linftyestimatesstrip}. By considering the fields $(\hat{\mathbf{A}}, \hat{\mathbf{E}}, \hat{\phi}, \hat{\pi})$ restricted to $\tau = \epsilon_{\mathrm{max}} - \tau_0/2$ as initial data and applying \Cref{thm:Linftyestimatesstrip}, one has that
\[ \| \hat{F}(\tau) \|_{L^\infty} + \| \hat{\D} \hat{\phi}(\tau) \|_{L^\infty} < \infty \]
for $\tau \leq \epsilon_{\mathrm{max}} + \tau_0/2$. But then the estimates \eqref{phiLinftyestimate} show that
\[ \| \hat{\phi}(\tau) \|_{L^\infty} + \| \hat{\mathbf{A}}(\tau) \|_{L^\infty} < \infty \]
for $\tau \leq \epsilon_{\mathrm{max}} + \tau_0/2$, and so by \eqref{energy12estimate} one deduces that $(\hat{\mathcal{E}}_1 + \hat{\mathcal{E}}_2)(\tau) < \infty$ up to $\tau = \epsilon_{\mathrm{max}} + \tau_0/2$. Since $(\hat{\mathcal{E}}_1 + \hat{\mathcal{E}}_2)^{1/2}$ is equivalent to the $(H^2 \times H^1)^2$ norm of $(\hat{\mathbf{A}}, \hat{\mathbf{E}}, \hat{\phi}, \hat{\pi})$, this contradicts the assumption that $\epsilon_{\mathrm{max}}$ was the maximal time of existence. Thus $\epsilon_{\mathrm{max}} = \infty$.
\end{proof}

\section{Asymptotics} \label{sec:asymptotics}

\subsection{De Sitter Space}

Recall that de Sitter space $\mathrm{dS}_4$ is the manifold $\mathbb{R} \times \mathbb{S}^3$ equipped with the metric 
\begin{equation} \label{deSittermetric} \tilde{g} = \d \alpha^2 - (\cosh^2 \alpha) \mathfrak{s}_3. \end{equation}
The vector field $\tilde{T}^a = \partial_\alpha$ is uniformly timelike and normal to surfaces of constant $\alpha$; we define the associated Riemannian metric $\tilde{\Gamma}$ on $\mathrm{dS}_4$ by
\[ \tilde{\Gamma}_{ab} = 2 \tilde{T}_a \tilde{T}_b - \tilde{g}_{ab}. \]
By making the change of variables
\[ \tan \frac{\tau}{2} = \tanh \frac{\alpha}{2}, \]
one finds that the de Sitter metric is conformal to the metric on the Einstein cylinder,
\[ \tilde{g} = \frac{1}{\cos^2 \tau} ( \d \tau^2 - \mathfrak{s}_3 ), \]
with the associated conformal factor $\omega = \cos \tau$. Under this conformal transformation $\mathrm{dS}_4$ is mapped to the section $(-\pi/2, \pi/2) \times \mathbb{S}^3$ of the Einstein cylinder, which puts the past and future null infinities of $\mathrm{dS}_4$ at
\[ \scri^- = \left\{ \tau = - \frac{\pi}{2} \right\} \times \mathbb{S}^3 \quad \text{and} \quad \scri^+ = \left\{ \tau = \frac{\pi}{2} \right\} \times \mathbb{S}^3. \]
Let us denote by $\tilde{\phi}$ and $\tilde{A}_a$ the scalar field and the Yang--Mills potential on de Sitter space. These are conformally related to the corresponding fields $\hat{\phi}$ and $\hat{A}_a$ on the Einstein cylinder by
\[ \hat{\phi} = \omega^{-1} \tilde{\phi}, \qquad \hat{A}_a = \tilde{A}_a. \]
It is clear that $(H^2 \times H^1)^2$ initial data on the hypersurface $\{ \alpha = 0 \}$ in de Sitter space defines $(H^2 \times H^1)$ initial data on $\{ \tau = 0\}$ in the Einstein cylinder. This follows from the fact that $\partial_\alpha$ is everywhere parallel to $\partial_\tau$, and the form of the conformal factor $\omega$. By \Cref{thm:largedataexistence}, we thus have a temporal gauge solution $(\hat{A}_a, \hat{\phi}) \in E_2(\mathbb{R} \times \mathbb{S}^3)^2$ on $\mathfrak{E}$, which is uniformly continuous on $I \times \mathbb{S}^3$ for any compact interval $I$. Indeed, this follows from the Sobolev embedding $H^2(\mathbb{S}^3) \hookrightarrow C^{0, \frac{1}{2}}(\mathbb{S}^3)$, which implies the inclusion
\[ E_2(I \times \mathbb{S}^3) \subset C^0(I \times \mathbb{S}^3). \]
Fixing the residual gauge freedom if necessary, we thus deduce that there exists a constant $c>0$ such that
\[ | \tilde{\phi} | \leq c \omega \leq c \e^{- |\alpha|}, \]
and, since $|\tilde{A}|^2_{\tilde{\Gamma}} = \omega^2 |\hat{A}|^2_{\hat{\Gamma}}$, also that
\[ |\tilde{A}|_{\tilde{\Gamma}} \leq c \omega \leq c \e^{-|\alpha|}. \]

\subsection{Minkowski Space}

Here we denote by $(A_a, \phi)$ the fields on Minkowski space $\mathbb{M}$, with the corresponding conformally related fields on the Einstein cylinder still denoted $(\hat{A}_a, \hat{\phi})$. Let $(\mathbf{a}, \mathbf{e}, \phi_0, \phi_1)$ be temporal gauge initial data for \eqref{Minkowskifieldequationstemporalgauge} satisfying the constraint \eqref{Minkowskiconstraint}, such that
\[ (\hat{\mathbf{a}}, \hat{\mathbf{e}}, \hat{\phi}_0, \hat{\phi}_1) = (\mathbf{a}, r_+^2\mathbf{e}, r_+^2 \phi_0, r_+^4 \phi_1) \in H^2(\mathbb{S}^3) \times H^1(\mathbb{S}^3) \times H^2(\mathbb{S}^3) \times H^1(\mathbb{S}^3). \footnote{Note that a sufficient condition is that $(\mathbf{a}, \mathbf{e}, \phi_0, \phi_1) \in (H^2_1(\mathbb{R}^3) \times H^1_2(\mathbb{R}^3))^2$, in the notation of \cite{ChoquetBruhatChristodoulou1981}.} \]
By construction, the data is such that it satisfies the hypotheses of \Cref{thm:largedataexistence}, giving a global temporal gauge solution $(\hat{A}_a, \hat{\phi})$ on the Einstein cylinder. This solution is related to the solution on Minkowski space by the usual scaling $\phi = \Omega \hat{\phi}$ and $A_a = \hat{A}_a$, where $\Omega = 2 (1+(t-r)^2)^{-1/2} (1+(t+r)^2)^{-1/2}$. Set $u=t-r$, $v=t+r$, $u_+ = \sqrt{1+u^2}$, and $v_+ = \sqrt{1+v^2}$. On $\mathbb{M}$ we have the tetrad
\[ l^a = - \partial_t + \partial_r = - 2 \partial_u, \quad n^a = \partial_t + \partial_r = 2 \partial_v, \quad e^a_\theta = \frac{1}{r} \partial_\theta, \quad e^a_\phi = \frac{1}{r \sin \theta} \partial_\phi, \]
with the metric expressed as
\[ \eta_{ab} = -\frac{1}{2}(l_a n_b + n_a l_b) + (e_A)_a (e_A)_b. \]
On $\mathfrak{E}$ we define the variables $\hat{u} = \tau - \zeta$, $\hat{v} = \tau + \zeta$, and the tetrad
\[ \hat{l}^a = - \partial_\tau + \partial_\zeta = - 2 \partial_{\hat{u}}, \hat{n}^a = \partial_\tau + \partial_\zeta = 2 \partial_{\hat{v}}, \quad \hat{e}^a_\theta = \frac{1}{\sin \zeta} \partial_\theta, \quad \hat{e}^a_\phi = \frac{1}{\sin \zeta \sin \theta} \partial_\phi, \]
in which the metric $\mathfrak{e}$ takes the form
\[ \mathfrak{e}_{ab} = -\frac{1}{2} ( \hat{l}_a \hat{n}_b + \hat{n}_a \hat{l}_b ) + (\hat{e}_A)_a ( \hat{e}_A)_b. \]
The relation between the two tetrads is
\begin{equation} \label{tetradrelations} l^a = \frac{2}{u^2_+} \hat{l}^a, \quad n^a = \frac{2}{v^2_+} \hat{n}^a, \quad \hat{e}^a_\theta = \Omega e_\theta^a, \quad \hat{e}^a_\phi = \Omega e^a_\phi, \end{equation}
where the Minkowski conformal factor is
\[ \Omega = \frac{2}{u_+ v_+}. \]
Using the conformal scaling of $\phi$, we then immediately deduce that
\[ | \phi | \leq c u_+^{-1} v_+^{-1}. \]
On the other hand, fixing the residual gauge freedom if necessary and using the relations \eqref{tetradrelations}, for the Yang--Mills potential we deduce
\[ | A_l | \leq c u_+^{-2} |\hat{A}_{\hat{l}}| \leq c u_+^{-2}, \qquad |A_n| \leq c v_+^{-2} |\hat{A}_{\hat{n}}| \leq c v_+^{-2}, \]
and
\[ |A|_{\mathfrak{s}_2} \leq c \Omega \leq c u_+^{-1} v_+^{-1}. \]
The above decay rates reproduce the decay rates of Yang and Yu \cite{YangYu2018}, requiring one fewer order of differentiability in the data. However, our results do not apply to the case of arbitrary charge at spatial infinity.

\pagebreak

\appendix

\section{An \texorpdfstring{$L^2$}{L2} bound for $\phi$ on the cone}

\begin{lemma} \label{lem:L2bound} The $L^2$ norm of $\phi$ on the cone $K(t_0)$ satisfies the bound
\[ \| \phi\|_{L^2(K(t_0))} \leq \| \phi \|_{L^2(B(r_0))}(-t_0) + 2 E^{1/2}_{\mathrm{loc}} t_0. \]
If moreover $\lambda \neq 0$, then
\[ \| \phi \|_{L^2(K(t_0))} \leq C E^{1/2}_{\mathrm{loc}} t_0^{3/4} ( 1 + t_0^{1/4} ). \]
\end{lemma}

\begin{proof} Since the bound is gauge independent, it suffices to prove it in the temporal gauge. Integrate $\nabla_a (|\phi|^2 K^a)$, $K^a = \partial_t$, over the region $\mathbf{K}(t_0)$ bounded by the past lightcone $K$ of the origin and the initial surface $\Sigma = \{ t = - t_0 \}$:
\[ \int_{\mathbf{K}(t_0)} \nabla_a ( |\phi|^2 K^a ) \, \d t \wedge \d^3 x = \int_{K(t_0)} | \phi |^2 \, r^2 \, \d r \, \d \Omega - \int_{B(r_0)\cap \Sigma} | \phi |^2 \, \d^3 x. \]
Now
\begin{align*} \left| \int_{\mathbf{K}(t_0)} \nabla_a ( |\phi|^2 K^a) \, \d t \wedge \d^3 x \right| & = \left| \int_{\mathbf{K}(t_0)} \partial_t ( |\phi|^2 ) \, \d t \wedge \d^3 x \right| \\
& \leq 2 \int_{-t_0}^{0} \int_{\mathbb{S}^2} \int_0^{-t} | \phi \cdot \pi |(t,r,\omega) \, r^2 \, \d r \, \d \Omega \, \d t \\
& \leq 2 \int_0^{t_0} \| \phi \|_{L^2(B(t))}(-t) \| \pi \|_{L^2(B(t))}(-t) \, \d t \\
& \leq 2 E_{\text{loc}}^{1/2} \int_0^{t_0} \| \phi \|_{L^2(B(r_0))}(-t) \, \d t,
\end{align*}
where we estimate the $L^2$ norm of $\phi$ on $B(r_0)$ by
\[ \frac{\d}{\d t} \| \phi \|^2_{L^2(B(r_0))} = 2 \int_{B(r_0)} \phi \cdot \pi \, \d^3 x \leq 2 \| \phi \|_{L^2(B(r_0))} E_{B(r_0)}^{1/2}. \]
This implies
\[ \frac{\d}{\d t} \| \phi \|_{L^2(B(r_0))}(t) \leq E_{\text{loc}}^{1/2}, \]
and so for $-t_0 \leq t \leq 0$
\[ \| \phi \|_{L^2(B(r_0))}(t) \leq \| \phi \|_{L^2(B(r_0))}(-t_0) + E_{\text{loc}}^{1/2} t_0. \]
Altogether then
\begin{align*} \| \phi \|^2_{L^2(K(t_0))} &\leq \| \phi \|^2_{L^2(B(r_0))}(-t_0) + 2 E^{1/2}_{\text{loc}} \int_0^{t_0} ( \| \phi \|_{L^2(B(r_0))}(-t_0) + E^{1/2}_{\text{loc}} t_0 ) \, \d t \\
& \leq \| \phi \|^2_{L^2(B(r_0))}(-t_0) + 4 E^{1/2}_{\text{loc}} \| \phi \|_{L^2(B(r_0))}(-t_0) t_0 + 4 E_{\text{loc}} t_0^2 \\
& \leq \left( \| \phi \|_{L^2(B(r_0))} (-t_0) + 2 E^{1/2}_{\text{loc}} t_0 \right)^2,
\end{align*}
which implies the first inequality. Now if $\lambda \neq 0$, since $B(r_0)$ is bounded we have
\[ \| \phi \|^2_{L^2(B(r_0))} \leq \| \phi \|^2_{L^4(B(r_0))} \frac{2}{\sqrt{3}} \sqrt{\pi} r_0^{3/2}, \]
so by \eqref{energyidentitycone}
\[ \| \phi \|_{L^2(B(r_0))} (-t_0) \leq C E_{\text{loc}}^{1/2} t_0^{3/4}. \]
Putting this into the first estimate completes the proof of the lemma.

\end{proof}

\pagebreak

\bibliographystyle{siam}
\bibliography{bibliography}

\begin{thebibliography}{10}

\bibitem{ChoquetBruhatChristodoulou1981}
{\sc Y.~Choquet-Bruhat and D.~Christodoulou}, {\em Existence of global
  solutions of the {Y}ang--{M}ills, {H}iggs and spinor field equations in $3+1$
  dimensions}, Annales scientifiques de l'\'Ecole Normale Sup\'erieure, Ser. 4,
  14 (1981), pp.~481--506.

\bibitem{ChoquetBruhatPaneitzSegal1983}
{\sc Y.~Choquet-Bruhat, S.~M. Paneitz, and I.~E. Segal}, {\em The
  {Y}ang--{M}ills equations on the universal cosmos}, Journal of Functional
  Analysis, 53 (1983), pp.~112--150.

\bibitem{ChruscielShatah1997}
{\sc P.~T. Chru\'sciel and J.~Shatah}, {\em Global existence of solutions of
  the {Y}ang--{M}ills equations on globally hyperbolic four dimensional
  {L}orentzian manifolds}, Asian J. Math., 1 (1997), pp.~530--548.

\bibitem{EardleyMoncrief1982a}
{\sc D.~M. Eardley and V.~Moncrief}, {\em The global existence of
  {Y}ang--{M}ills--{H}iggs fields in $4$-dimensional {M}inkowski space. {I}.
  local existence and smoothness properties}, Comm. Math. Phys., 83 (1982),
  pp.~171--191.

\bibitem{EardleyMoncrief1982b}
\leavevmode\vrule height 2pt depth -1.6pt width 23pt, {\em The global existence
  of {Y}ang--{M}ills--{H}iggs fields in $4$-dimensional {M}inkowski space.
  {II}. completion of proof}, Comm. Math. Phys., 83 (1982), pp.~193--212.

\bibitem{Friedlander1975}
{\sc F.~G. Friedlander}, {\em The wave equation on a curved space-time.},
  Cambridge University Press, Cambridge, UK, 1975.

\bibitem{Ghanem2016}
{\sc S.~Ghanem}, {\em The global non-blow-up of the {Y}ang--{M}ills curvature
  on curved space-times}, Journal of Hyperbolic Differential Equations, 13
  (2016), pp.~603--631.

\bibitem{GinibreVelo1981}
{\sc J.~Ginibre and G.~Velo}, {\em The {C}auchy problem for coupled
  {Y}ang--{M}ills and scalar fields in the temporal gauge}, Comm. Math. Phys.,
  82 (1981), pp.~1--28.

\bibitem{GoganovKapitanskii1985}
{\sc M.~V. Goganov and L.~V. Kapitanskii}, {\em Global solvability of the
  {C}auchy problem for the {Y}ang--{M}ills--{H}iggs equations}, Zap. Nauchn.
  Sem. Leningrad. Otdel. Mat. Inst. Steklov. (LOMI), 37 (1985), pp.~18--48.

\bibitem{JaffeTaubes1980}
{\sc A.~Jaffe and C.~Taubes}, {\em Vortices and Monopoles: Structure of Static
  Gauge Theories}, Birkhäuser Verlag, Boston-Basel-Stuttgart, 1980.

\bibitem{KlainermanMachedon1993}
{\sc S.~Klainerman and M.~Machedon}, {\em Space-time estimates for null forms
  and the local existence theorem}, Communications on Pure and Applied
  Mathematics, 46 (1993), pp.~1221--1268.

\bibitem{KlainermanMachedon1994}
{\sc S.~Klainerman and M.~Machedon}, {\em On the {M}axwell--{K}lein--{G}ordon
  equation with finite energy.}, Duke Math. J., 74 (1994), pp.~19--44.

\bibitem{KlainermanMachedon1995}
\leavevmode\vrule height 2pt depth -1.6pt width 23pt, {\em Finite energy
  solutions of the {Y}ang--{M}ills equations in $\mathbb{R}^{3+1}$}, Annals of
  Mathematics, 142 (1995), pp.~39--119.

\bibitem{KlainermanRodnianski2006}
{\sc S.~Klainerman and I.~Rodnianski}, {\em A {K}irchoff--{S}obolev parametrix
  for the wave equation and applications}, arXiv:math/0603009,  (2006).

\bibitem{SungJinOh2015}
{\sc S.-J. Oh}, {\em Finite energy global well-posedness of the {Y}ang--{M}ills
  equations on $\mathbb{R}^{1+3}$ : An approach using the {Y}ang--{M}ills heat
  flow}, Duke Math.J., 164 (2015), pp.~1669--1732.

\bibitem{Segal1963}
{\sc I.~Segal}, Ann. Math., 78 (1963).

\bibitem{Segal1979}
\leavevmode\vrule height 2pt depth -1.6pt width 23pt, J. Funct. Anal., 33
  (1979).

\bibitem{SelbergTesfahun2013}
{\sc S.~Selberg and A.~Tesfahun}, {\em Null structure and local well-posedness
  in the energy class for the {Y}ang--{M}ills equations in {L}orenz gauge},
  arXiv:1309.1977,  (2013).

\bibitem{Tao2003}
{\sc T.~Tao}, {\em Local well-posedness of the {Y}ang--{M}ills equation in the
  temporal gauge below the energy norm}, Journal of Differential Equations, 189
  (2003), pp.~366--382.

\bibitem{Taujanskas2018}
{\sc G.~Taujanskas}, {\em Conformal scattering of the {M}axwell-scalar field
  system on de {S}itter space}, arXiv:1809.01559,  (2018).

\bibitem{YangYu2018}
{\sc S.~Yang and P.~Yu}, {\em On global dynamics of the
  {M}axwell--{K}lein--{G}ordon equations}, arXiv:1804.00078,  (2018).

\end{thebibliography}

\end{document}